\documentclass[onefignum,onetabnum,oneeqnum,onethmnum]{siamart190516}
\pdfoutput=1
\usepackage{fullpage}
\usepackage[title]{appendix}
\usepackage{graphicx}
\usepackage[normalem]{ulem}
\graphicspath{{figures/}}
\usepackage{amsmath}
\usepackage{amsfonts}
\usepackage{amssymb}
\usepackage{cleveref}
\usepackage{xcolor}
\usepackage{algorithm,algpseudocode}
\algblock{Input}{EndInput}
\algnotext{EndInput}
\algblock{Output}{EndOutput}
\algnotext{EndOutput}
\newcommand{\Desc}[2]{\State \makebox[9em][l]{#1}#2}

\definecolor{grey}{rgb}{0.5, 0.5, 0.49}

\usepackage{verbatim}
\usepackage{longtable,tabularx,mathtools,leftidx,multirow}
\usepackage{epstopdf}
\usepackage{pgfplots}
\usepackage{subcaption}

\marginparwidth=\dimexpr \marginparwidth - 1.1cm\relax

\newsiamthm{problem}{Problem} 
\newcommand{\mcaption}[1]{\caption{{\em\small #1}}}

\newcommand{\R}[1]{\mathbb{R}^{#1}}

\newcommand{\mat}[1]{\begin{bmatrix}#1\end{bmatrix}}
\newcommand{\tenprod}[2]{{_{#1}{\times}_{#2}}}

\definecolor{rackhamgreen}{HTML}{75988d}
\definecolor{tappanred}{HTML}{9a3324}
\definecolor{wavefieldgreen}{HTML}{a5a508}
\definecolor{matthaeiviolet}{HTML}{577294}

\newcommand{\newdim}{n_{d+1}}

\newcommand{\reals}{\mathbb{R}}
\DeclareMathOperator*{\argmin}{argmin}

\title{An Incremental Tensor Train Decomposition Algorithm} 

\author{Doruk Aksoy\thanks{Department of Aerospace Engineering, University of Michigan, Ann Arbor, MI, USA} (\email{doruk@umich.edu}, \email{goroda@umich.edu}) \and David J. Gorsich\thanks{Ground Vehicle Systems Center, U.S. Army, Warren, MI, USA} (\email{david.j.gorsich.civ@army.mil}) \and
	 Shravan Veerapaneni\thanks{Department of Mathematics, University of Michigan, Ann Arbor, MI, USA} (\email{shravan@umich.edu})\\ DISTRIBUTION A. Approved for public release; distribution unlimited. OPSEC $\# 7057$
	\and Alex A. Gorodetsky\footnotemark[1]
}

\begin{document}
	\maketitle
	\begin{abstract}
		We present a new algorithm for incrementally updating the tensor train decomposition of a stream of tensor data. This new algorithm, called the {\em tensor train incremental core expansion} (TT-ICE) improves upon the current state-of-the-art algorithms for compressing in tensor train format by developing a new adaptive approach that incurs significantly slower rank growth and guarantees compression accuracy. This capability is achieved by limiting the number of new vectors appended to the TT-cores of an existing accumulation tensor after each data increment. These vectors represent directions orthogonal to the span of existing cores and are limited to those needed to represent a newly arrived tensor to a target accuracy. We provide two versions of the algorithm: TT-ICE and TT-ICE accelerated with heuristics (TT-ICE$^*$). We provide a proof of correctness for TT-ICE and empirically demonstrate the performance of the algorithms in compressing large-scale video and scientific simulation datasets. Compared to existing approaches that also use rank adaptation, TT-ICE$^*$ achieves $57\times$ higher compression and up to $95\%$ reduction in computational time. 
	\end{abstract}



	\section{Introduction}\label{sec:intro}

	Tensors provide a representation for multivariate or high-dimensional data in many problems, including RGB images~\cite{Jia2014}, social networks~\cite{Sizov2010,Nakatsuji2017}, multisensory experiments~\cite{Anaissi2020}, neuroscience~\cite{Martinez-Montes2004,Miwakeichi2004}, and finance~\cite{Giudici2019}. As the size of the tensor increases, computational processes become harder, if not impossible, due to the curse of dimensionality --- both storage and computational processing requirements can grow exponentially with the number of dimensions. Tensor decompositions that rely on low-rank approximations provide a solution to mitigate the curse of dimensionality to help reduce computational demands.

In many applications, data becomes available incrementally, e.g., from a sequence of experiments~\cite{Anaissi2020},  a video (which can be seen as a sequence of RGB images)~\cite{Jia2014}, in IoT applications~\cite{cook2019anomaly}, or from simulations of physical systems~\cite{Alben2019}. As a result, batch tensor compression approaches that wait for all data to be collected may not be affordable because the data storage may exceed capacity. Furthermore, the total number of data points might not be known {\em a priori}. In this setting, it would be inefficient to reconstruct and decompose the tensor every time a new data point arrives. Incremental algorithms that update the decomposition without reconstructing the compressed tensor are needed.
 
 Existing works include incremental tensor decomposition algorithms in Canonical Polyadic (CP)  format~\cite{Anaissi2020,Du2018,Vandecappelle2017,Smith2018}, Tucker format~\cite{Yu2015,Sobral2014,Baker2012}, and tensor train (TT)~\cite{Oseledets2011} format~\cite{Thanh2021,Liu2018,Wang2021}. In this work, we focus on the TT-format. Unlike the CP-format, the TT-format avoids the NP-hard problem of computing the canonical rank \cite{Hastad1989} and offers controlled compression algorithms. Also, the storage of a tensor in TT-format scales linearly in the number of dimensions $d$, whereas in Tucker format, it scales exponentially in $d$. Furthermore, the TT-format offers an efficient way to execute
 basic linear algebra operations without reconstructing the full tensor, and therefore it is beneficial for computing with large-scale data.
 
 In this work, we consider a \textit{$d$-way tensor} $\mathcal{Y}\in\R{n_1\times\cdots\times n_d}$ to be a multidimensional array. 
 A \textit{tensor stream} (or alternatively, \textit{stream of tensors}) is a sequence of $d$-way tensors $\mathcal{Y}^1,\mathcal{Y}^2$,\dots, where each element in the sequence $\mathcal{Y}^k\in \R{n_1\times\cdots\times n_d}$ is a $d$-dimensional tensor\footnote{Despite the conservative definition here, we generalize the notion of tensor stream to support batches of tensors in \Cref{sec:methodology}.}. A finite stream of tensors is called an \textit{accumulation} and can be viewed as a $(d+1)$-way tensor $\mathcal{X}^k\in\R{n_1\times\cdots\times n_d\times \newdim^{k}}$ by concatenating the tensors in the stream along the last dimension. The problem we consider can then be summarized as follows.

 \begin{problem}\label{prb} 
 Construct a scheme to update the approximation $\hat{\mathcal{X}}^{k}$ of the accumulation tensor $\mathcal{X}^{k}$ after every increment $k$ in TT-format. The constructed scheme should maintain the guaranteed bounds on the error $\|\mathcal{X}^k-\hat{\mathcal{X}}^k\|_{F}$ for all $k$. Furthermore, this approximate accumulation should represent all tensor increments $\mathcal{Y}^{\ell}$ with error $\|\mathcal{Y}^{\ell}-\hat{\mathcal{Y}}^{\ell}\|_{F}\leq\varepsilon_{des}\|\mathcal{Y}^{\ell}\|_{F}$ for any $\ell\leq k$, where $\hat{\mathcal{Y}}^{\ell}$ can be extracted from $\hat{\mathcal{X}}^{k}$.
 \end{problem}
 
There are a few recent works for computing an incremental TT in this setting~\cite{Thanh2021,Liu2018}; however, they suffer from drawbacks that limit their scalability. For example, the TT-FOA~\cite{Thanh2021} algorithm considers a fixed-rank approximation and uses an optimization procedure that cannot guarantee bounded errors on {\it all} tensors seen in the stream. The ITTD algorithm~\cite{Liu2018} (also see \cite{de2022tensor}), on the other hand, can guarantee a prescribed error tolerance and enables rank adaptation. However, this algorithm has computational challenges, which we show are due to inefficient preprocessing of new data and overly conservative rank growth when the new data is incorporated into the existing representation. Moreover, it does not maintain an orthonormal set of TT-cores and therefore cannot be used to project arbitrary new data to obtain its latent representation in TT-format.

	Our contribution is a new incremental algorithm, the TT incremental core expansion (TT-ICE), for computing and updating the tensor train representation of an approximate accumulation $\hat{\mathcal{X}}^k$. Moreover, this approach is suitable for the online setting and never requires full storage or reconstruction of all data while providing a solution to \Cref{prb}. Specifically, the new aspects of our approach include:
		\begin{itemize}
    	\item efficiently updating the rank to accurately represent each new data point;
		\item controlling error growth and maintaining provably guaranteed bounds on the compression error of each tensor in the stream for all times, and 
		\item a set of heuristic modifications that further significantly increase computational efficiency -- this algorithm is called TT-ICE$^*$.
	\end{itemize} 

These contributions are achieved by limiting the number of new vectors appended to the TT-cores of an existing accumulation tensor after each data increment. These vectors represent directions orthogonal to the span of the existing cores and are limited to those needed to represent a newly arrived tensor to a target error.
	Moreover, our theoretical results are empirically justified on video compression applications arising from video game data and from solutions to numerical PDEs. Our results indicate order-of-magnitude benefits in compression by TT-ICE over ITTD. We also demonstrate that TT-ICE works in cases where ITTD runs out of memory. Finally, we show that the heuristic version, TT-ICE$^*$, can achieve $2.6-7.3\times$ speedup over TT-ICE, with only a negligible sacrifice in compression accuracy.
	
	The rest of this paper is structured as follows. In Section~\ref{sec:background}, we present the foundational concepts behind the TT-format and discuss the existing literature on incremental TT-decompositions, in detail. In Section~\ref{sec:methodology}, we present the TT-ICE and TT-ICE$^*$ algorithms and prove the correctness of TT-ICE. In Section~\ref{sec:experiments}, we provide preliminary experiments using our proposed approach on physical and cyber-physical data.

	\section{Background}
	\label{sec:background}
	In this section, we provide the necessary background for the TT-decomposition and review two existing approaches for incremental approximation in TT-format.
	
	\subsection{tensor train decomposition} \label{sec:tt}
	This section reviews the relevant background on tensors and the tensor train decomposition. 
	
	The mode-$i$ \textit{matricization} (or \textit{unfolding}) of a $d$-way tensor $\mathcal{A}\in\R{n_1\times\cdots\times n_d}$ reshapes the tensor into a matrix $A_{(i)}$ with size $n_1\dots n_{i}\times n_{i+1}\dots n_d$.  In this unfolding, all the modes up to the $i$-th mode are mapped into the rows of the matrix and all other modes are mapped to the columns.

	The \textit{contraction} between two tensors $\mathcal{A}\in\R{n_1\times\cdots\times n_d}$ and $\mathcal{B}\in\R{n_d\times n_{d+1}\times\cdots\times n_{D}}$ along the $d$-th dimension of $\mathcal{A}$ and the first dimension of $\mathcal{B}$ is a binary operation represented as
	\begin{equation}
		\mathcal{C} = \mathcal{A}\ \tenprod{d}{1}\ \mathcal{B}, \quad \textrm{ where } \quad
		\mathcal{C}\left(i_{1},i_{2},\dots,i_{d-1},i_{d+1},\dots,i_{D}\right)=\sum_{j=1}^{n_d}\mathcal{A}\left(i_{1},\dots,i_{d-1},j\right)\mathcal{B}\left(j,i_{d+1},\dots,i_{D}\right)
	\end{equation}
	and the output $\mathcal{C}\in\R{n_1\times\cdots\times n_{d-1}\times n_{d+1}\times\cdots\times n_{D}}$ becomes a $(D-2)$-way tensor. The subscripts on either side of the $\times$ sign indicate the contraction axes of the tensors on their respective sides.
	
	A $d$-way tensor $\mathcal{Y}\in\R{n_1\times\cdots\times n_d}$ is said to be in the \textit{TT-format} when it is represented by a sequence of contractions of 3-way tensors $\mathcal{G}_i\in\R{r_{i-1}\times n_i\times r_{i}}$, for $i=1,\dots,d$, according to
	\begin{equation}
		\mathcal{Y}=\mathcal{G}_1\ \tenprod{3}{1}\ \mathcal{G}_2\ \tenprod{3}{1}\ \cdots \ \tenprod{3}{1} \ \mathcal{G}_d.
		\label{eq:tensorcontraction}
	\end{equation}
	The tensors $\mathcal{G}_i$ are called \textit{TT-cores} and the $r_i$, $i=0,\dots,d$ are called \textit{TT-ranks} with $r_0=r_d=1$. 
	The TT-ranks are equal to the ranks of a sequential set of unfoldings~\cite{Oseledets2011}. Specifically, the $i$-th TT-rank $r_i$ is equal to the rank of the mode-$i$ unfolding matrix $Y_{(i)}$. 
	In practice, the unfolding matrices are rarely exactly low-rank, and instead, an \textit{approximate} TT-representation is computed.

 In this case, a rank-$r_i$ truncated SVD of the $i$-th unfolding satisfies

\begin{equation}
		Y_{(i)}=U_i \Sigma_i V_i + E_i,
		\label{eq:deltasvd}
	\end{equation}
 	where $U_i\in\R{m\times r_i}$, $\Sigma_i\in\R{r_i\times r_i}$, $V_i\in\R{r_i\times \ell}$ are the factors of a rank-$r_i$ truncated SVD, and $E_i \in \reals^{m \times \ell}$ is the residual. The dimensions of the unfolding require $m=\prod_{j=1}^{i}n_j$, and $\ell=\prod_{j=i+1}^{d}n_j$. The Eckart-Young-Mirsky Theorem guarantees that the truncated SVD provides the best rank-$r_i$ approximation of the matrix in the Frobenius norm, and the reconstruction error can be bounded by the remaining singular values according to $	\|Y_{(i)}-U_i\Sigma_i V_i\|_F=\|E_i\|_F=\sqrt{\sum_{j=r_{i}+1}^{\min(m,n)} ( \sigma_{j}^{(i)})^{2}},$ where $\sigma_{j}^{(i)}$ is the $j$-th singular value of the mode-$i$ unfolding~\cite{Eckart1936,MIRSKY1960}.
	
	The proof for the TT-ranks~\cite[Th.~2.1]{Oseledets2011}, is constructive and provides an algorithm, called TT-SVD~\cite[Alg.~1]{Oseledets2011}, to compress a tensor to a target accuracy. For a $d$-way tensor $\mathcal{Y}\in\R{n_1\times\cdots\times n_d}$, TT-SVD begins with a truncated SVD of the mode-1 unfolding $Y_{(1)}\in\R{n_1\times n_2\dots n_d}$
	\begin{equation}
		Y_{(1)}=U_1\Sigma_1 V_1+E_1,
		\label{eq:ttcoresstep1}
	\end{equation}
	with orthonormal left singular vectors $U_1\in\R{n_1\times r_1}$ and orthogonal $\Sigma_1 V_1\in\R{r_1\times n_2\dots n_d}$. Then, $\Sigma_1 V_1$ needs to be compressed. This process uses the fact that orthonormality of the left singular vectors guarantees $U_1^TE_1=0,$ so that multiplying $Y_{(1)}$ on the left by $U_1$ leads to 
	\begin{equation}
		U_1^TY_{(1)}=\Sigma_1 V_1.
		\label{eq:ttcoresstep2}
	\end{equation}
	Now, the matrix $\Sigma_1 V_1$ can be reshaped to form $\mathcal{Z}\in\R{r_1n_2\times n_3\times\cdots\times n_d}$, and a  truncated SVD is computed for its mode-1 unfolding
	\begin{equation}
		Z_{(1)}=U_2\Sigma_2 V_2+E_2,
		\label{eq:ttcoresstep3}
	\end{equation} 
	with $U_2\in\R{r_1 n_2\times r_2}$ and $\Sigma_2 V_2\in\R{r_2\times n_3\dots n_d}$. The process then repeats analogous to \Cref{eq:ttcoresstep2,eq:ttcoresstep3} until all dimensions are considered. The left singular vectors $U_i\in\R{r_{i-1}n_i\times r_i}$ of each decomposition become the first $d-1$ TT-cores through reshaping
	\begin{equation}\label{eq:tt_core_reshape}
		\mathcal{G}_i = \texttt{reshape}\left(U_i, \left[r_{i-1}, n_{i}, r_{i}\right]\right) \in\R{r_{i-1}\times n_i\times r_i}.
	\end{equation}
	The final core consists of the right singular vectors scaled with their respective singular values from the final truncated SVD.
	
	The truncation errors for each unfolding described above need to be carefully chosen to ensure a guaranteed relative error bound $\varepsilon\in[0,1]$ with respect to the Frobenius norm.
{Specifically, choosing a dimension-dependent truncation error according to}
$	\delta=\frac{\varepsilon}{\sqrt{d-1}}\|\mathcal{Y}\|_F$
ensures that the computed TT-approximation $\hat{\mathcal{Y}}$ has a relative error less than $\varepsilon$, i.e.  $\|\mathcal{Y}-\hat{\mathcal{Y}}\|_F\leq\varepsilon\|\mathcal{Y}\|_F$ ~\cite[Thm.~2.2]{Oseledets2011}.

Once a TT is computed, basic algebraic operations, such as addition and multiplication, can be computed in closed form. For example, two $d$-way tensors in TT-format with cores $\{\mathcal{G}_i\}_{i=1}^{d}$ and $\{\mathcal{H}_{i}\}_{i=1}^{d}$ can be added if both of the tensors have the size $n_1\times\dots\times n_d$. The resulting tensor $\mathcal{D}=\mathcal{G}+\mathcal{H}$ is calculated using the following rule
\begin{equation}
	\mathcal{D}_m(j_m)=
	\begin{cases}\vspace{0.2ex}
		\mat{\mathcal{G}_1(j_1)&\mathcal{H}_1(j_1)} &,\ m=1;\\ \vspace{0.2ex}
		\mat{\mathcal{G}_m(j_m)& 0 \\ 0& \mathcal{H}_m(j_m)}& ,\ m=2,\dots d-1;\\ 
		\mat{\mathcal{G}_d(j_d)\\\mathcal{H}_d(j_d)}&,\ m=d,
	\end{cases}
	\label{eq:ttaddition}
\end{equation}
where $\mathcal{G}_m(j_m)\in \R{r_{m-1}\times r_m}$ denotes the $j_m$-th slice of the $m$-th TT-core $\mathcal{G}_m$. If $\mathcal{G}$ has TT-ranks $r_i$ and $\mathcal{H}$ has TT-ranks $\bar{r}_i$, the sum $\mathcal{D}$ will have the TT-ranks $\hat{r}_{i}=r_i+\bar{r}_i$ with the exception of $\hat{r}_0=\hat{r}_d=1$.

\subsection{Existing incremental tensor decompositions}\label{sec:existingliterature}

This section discusses two existing approaches for solving Problem~\ref{prb}: the First-Order Adaptive Tensor Train Decomposition (TT-FOA)~\cite{Thanh2021} and the Incremental Tensor Train Decomposition (ITTD)~\cite{Liu2018}.

There are two main steps of incremental procedures: (1) preprocessing the newly arrived tensor $\mathcal{Y}^{k+1}\in\R{n_1\times \cdots\times n_d}$ and (2) updating the TT-cores of the previous accumulation tensor $\mathcal{X}^{k}$ to obtain new TT-cores for  $\mathcal{X}^{k+1}$ with this information. 
Below we describe how TT-FOA and ITTD handle these two steps.

\paragraph{Preprocessing $\mathcal{Y}^{k+1}$}
At time $k+1$, TT-FOA poses the problem of preprocessing the new tensor $\mathcal{Y}^{k+1}$ using the TT-cores of the approximate accumulation $\hat{\mathcal{X}}^{k}$ as the following regularized optimization problem
\begin{equation}
	g^{k+1}=\argmin_{g\in\R{r_d\times1}} \|\mathcal{Y}^{k+1}-\hat{\mathcal{X}}^k\ \tenprod{(d+1)}{1}g \|_{F}^{2} +\frac{\rho}{2} \|g\|_2^2,
	\label{eq:tt-foaoptimization}
\end{equation}
where $g^{k+1}$ is the representation of $\mathcal{Y}^{k+1}$ as the $(d+1)$-th TT-core of the accumulation $\mathcal{X}^{k}$, $\rho$ is a small regularization parameter, and the contraction of first $d$ TT-cores $\hat{\mathcal{X}}^k\in\R{n_1\times\cdots\times n_d\times r_d}$ is computed as
\begin{equation}
	\hat{\mathcal{X}}^k =\mathcal{G}_1\ \tenprod{3}{1}\ \cdots \ \tenprod{3}{1}\ \mathcal{G}_{d}.
	\label{eq:tt-foastep1}
\end{equation} 
Note that these first $d$ cores are common for all compressed data, and only the final core is unique to a particular datum. In other words, one can view the first $d$ cores as the basis of approximation that is common to all the data and the final core as the coefficients of this basis that are specific to each datum. 
Without the regularization term, the optimal solution $g^{k+1}$ would be the projection of $\mathcal{Y}^{k+1}$ onto the basis defined by $\hat{\mathcal{X}}^k$. Thus, any component of the new data not represented by the existing accumulation would be discarded. The regularization biases the solution to be closer to zero.

In \cite{Thanh2021}, the optimization problem is solved using the randomized sketching technique \cite{Mahoney2011}. The closed-form solution is
\begin{equation}\label{eq:tt-foaridgeregression}
	g^{k+1}=\left(\mathcal{L}(\hat{X}_{(d)}^{k})^T\mathcal{L}(\hat{X}_{(d)}^{k})+\rho I_{r_d}\right)^{-1}\mathcal{L}(\hat{X}_{(d)}^{k})^T\mathcal{L}(y^{k+1}),
\end{equation}
with $\hat{X}_{(d)}^{k}\in\R{n_1\dots n_d\times r_d}$ as the mode-$d$ unfolding of $\hat{\mathcal{X}}^k$, $y^{k+1}\in\R{n_1\dots n_d}$ as the reshaping of $\mathcal{Y}^{k+1}$ into a vector, and $\mathcal{L}(\cdot)$ as the sketching map that samples the same set of rows from $\hat{X}_{(d)}^{k}$ and $y^{k+1}$.

ITTD follows another approach. It reshapes the streamed $d$-way tensor $\mathcal{Y}^{k+1}$ to be a $d+1$-way tensor with size $n_1\times\cdots\times n_d\times 1$ and computes the TT-decomposition of the reshaped tensor using the TT-SVD algorithm, without leveraging any information from the accumulation tensor. As a result, it requires the computation of a new sequence of singular value decompositions to discover an entirely new basis --- a computation step not needed by TT-FOA or our proposed approach.

\paragraph{Updating the TT-approximation of $\mathcal{X}^{k}$ to that of $\mathcal{X}^{k+1}$}

TT-FOA assumes fixed TT-ranks and adopts a gradient descent procedure to update the TT-cores of $\mathcal{X}^{k}$. With this fixed-rank assumption, TT-FOA cannot guarantee a predefined upper bound for representation error except in the limit. TT-FOA uses the following optimization objective to update the TT-cores of $\mathcal{X}^{k}$

\begin{equation}\label{eq:tt-foacoreoptimization1}
\mathcal{G}_i=\argmin_{\mathcal{G}\in\R{r_{i-1}\times n_i\times r_i}}\sum_{j=1}^{k+1}\lambda^{k+1-j}\|\mathcal{Y}^j-\mathcal{A}_i\ \tenprod{(i+1)}{1}\ \mathcal{G}\ \tenprod{(i+2)}{1}\ \mathcal{B}_{i}^{j} \|^2_F,
\end{equation}
where
\begin{equation}\label{eq:tt-foacorecontraction}
	\begin{split}
		\mathcal{A}_{i}&=\mathcal{G}_{1}\ \tenprod{3}{1} \ \cdots \ \tenprod{3}{1}\ \mathcal{G}_{i-1}, \\
		\mathcal{B}_{i}^{j}&=\mathcal{G}_{i+1}\ \tenprod{3}{1} \ \cdots \ \tenprod{3}{1}\ \mathcal{G}_d \ \tenprod{3}{1} \ g^{j},
	\end{split}
\end{equation}
and
$\lambda\in(0,1]$ is the forgetting factor discounting the effect of previous tensors. Note that only the superscript of the forgetting factor $\lambda$ is for exponentiation. $g^j$ denotes the coefficient vector representing $\mathcal{Y}^j$ in the $(d+1)$-th TT-core. Note that the objective function updates the $i$-th core $\mathcal{G}_i$ independent from the other cores and allows updating the cores in parallel. The TT-FOA algorithm provides a recursive approach to efficiently update this core, and we leave further details to~\cite{Thanh2021}.


Next, we describe how ITTD updates the accumulated tensor. Recall that at ITTD, this update is preceded by reshaping the $d$-way tensor $\mathcal{Y}^{k+1}$ into a $d+1$-way tensor with the $(d+1)$-th dimension as 1 and calculating an independent TT-decomposition for $\mathcal{Y}^{k+1}$. ITTD then merges the newly compressed tensor with the existing accumulation through a specific addition operation. To this end, the TT-core of $\mathcal{Y}^{k+1}$ corresponding to the growing dimension is padded with zeros so that the growing dimension becomes {$(k+1)$}. The same padding procedure is repeated for the TT-core of $\mathcal{X}^k$ so that the growing dimension becomes $(k+1)$. This zero-padding process is done to ensure the dimensional consistency between the two TT-representations. While padding zeros to the TT-cores, the ITTD algorithm places the padding tensors so that both tensors can be added without interfering with each other. ITTD algorithm then combines the TT-cores of the streamed tensor $\mathcal{Y}^{k+1}$ and the existing TT-cores of $\mathcal{X}^k$ by adding them in TT-format using the rule shown in \eqref{eq:ttaddition}.
After the addition of these two tensors in TT-format, the TT-ranks are the sum of the component ranks. Therefore, an optional final step of the ITTD algorithm, the TT-rounding procedure of \cite[Alg 2]{Oseledets2011} is executed to reorthogonalize and recompress the TT-cores after addition.

\subsection{Limitations of existing approaches}\label{sec:limitations}
This section discusses the shortcomings of the existing methods in the literature.

The main limitation of TT-FOA is that the TT-representation error cannot be guaranteed for two reasons. First, due to the random initialization and the recursive update scheme, the TT-FOA algorithm starts with a high representation error. As time progresses and new observations become available, this error converges to a steady-state value. However, the lower bound of this steady-state value depends on the second limitation -- that the TT-ranks of the accumulation are fixed as an input to the TT-FOA algorithm. If the selected TT-ranks are insufficient to represent the data accurately, there is no choice but to restart the entire TT-FOA algorithm with higher TT-ranks. Furthermore, the forgetting factor influences the TT-cores of the accumulation, so that representing the recently streamed tensors is more important than the earlier parts of the accumulation. This results in a loss of representation accuracy if the information in the tensor stream varies over time.

On the other hand, the approach we propose does not require time to converge to a steady-state error for an accumulation since truncated SVD computes the best-low-rank approximation. Furthermore, when the current TT-ranks of the decomposed accumulation are not sufficient to represent the streamed tensor $\mathcal{Y}^{k+1}$ up to a prescribed precision, our approach increases the TT-ranks accordingly to meet the precision criterion. Most importantly, our approach guarantees that the whole accumulation is represented within their prescribed precision bounds for all times.

The limitation of the ITTD algorithm stems from the inefficiencies of using the existing TT-approximation of accumulation $\mathcal{X}^k$ to aid the incorporation of the new data tensor. If the streamed data is highly structured, the TT-SVD step in ITTD is likely to return TT-cores that are redundant with the TT-cores of the accumulation. Therefore, the addition in TT-format likely results in linearly dependent TT-cores and superfluous TT-ranks that make TT-rounding an obligatory step for adequate compression.

In the case of an unbounded tensor stream, such a rounding procedure is required; otherwise, the TT-cores would grow without bound. However, the rounding procedure has a complexity that scales with the cube of the TT-ranks. This leads to a vicious cycle between rank growth and speed for ITTD. Putting off TT-rounding allows faster execution time, but the rank growth causes a cubic increase in time for future rounding attempts. On the other hand, avoiding rounding results in unbounded rank growth, which makes the rounding procedure impossible due to memory limitations.

We provide a more controlled rank growth by using the TT-cores of $\mathcal{X}^k$ as a foundation and expanding them only with complementing information obtained through $\mathcal{Y}^{k+1}$.
Furthermore, for the edge case where the current TT-cores of the accumulation can represent the streamed data sufficiently accurate, our method can skip updating the TT-cores of the accumulation.

In the next section, we present our incremental TT algorithm, which provides a solution to all the shortcomings of existing incremental TT-decomposition algorithms in detail.

\section{Methodology: the TT-ICE algorithm}
\label{sec:methodology}
In this section, we present a new incremental TT-decomposition algorithm. Similar to the existing work, we tackle the problem of computing the incremental TT-decomposition of tensor streams in two steps: preprocessing and update. However, a distinguishing feature of our approach is that it processes and updates each dimension sequentially in a single pass. Furthermore, the subsequent algorithm considers a slightly generalized setting in which $\newdim^{k+1}$ tensors become available during the $k+1$-th increment. Equivalently, this can view the new data at increment $k+1$ as a tensor with an additional dimension $\mathcal{Y}^{k+1}\in\R{n_1\times\cdots\times n_d\times \newdim^{k+1}}$.

\subsection{Overview}
This section walks through the steps of TT-ICE and presents a pseudocode for TT-ICE.

The proposed algorithm updates the TT-cores $\{\mathcal{G}_{i}^k\}_{i=1}^{d+1}$ of the approximate accumulation tensor $\hat{\mathcal{X}}^k$ to obtain updated cores $\{\mathcal{G}_{i}^{k+1}\}_{i=1}^{d+1}$ of the approximate accumulation tensor $\hat{\mathcal{X}}^{k+1}.$  Let $\{U_{i}^k\}_{i=1}^{d+1}$ denote the reshapings of the cores at increment $k$ according to Equation~\eqref{eq:tt_core_reshape}, these will then be updated to obtain $\{U_{i}^{k+1}\}_{i=1}^{d+1}.$ At each increment, these matrices will be orthonormal, and the update will ensure they remain orthonormal.
Similarly, we will update the ranks $\{r_{i}^{k}\}_{i=0}^{d+1}$ to $\{r_{i}^{k+1}\}_{i=0}^{d+1}.$  The approach is provided in Algorithm~\ref{alg:TT-ICE}, and we describe each step next.

\begin{algorithm}
	\caption{\texttt{TT-ICE}: Incremental update of a tensor train decomposition}
	\label{alg:TT-ICE}
	\begin{algorithmic}[1]
		\Input 
		\Desc{$\left\{U^{k}_i\right\}_{i=1}^{d+1}$}{$\quad$ reshaped cores of the TT-decomposition of the accumulation $\mathcal{X}^k$} 
		\Desc{$\mathcal{Y}^{k+1}\in\R{n_1\times\cdots\times n_d\times \newdim^{k+1}}$}{$\quad$ new tensor}
		\Desc{$\left\{\epsilon_{i}\right\}_{i=1}^{d}$}{$\quad$ SVD truncation tolerances}
		\EndInput
		\Output 
		\Desc{$\left\{U^{k+1}_i\right\}_{i=1}^{d+1}$}{ $\quad$ updated cores for the accumulation $\mathcal{X}^{k+1}$} 
		\EndOutput

		\State \textcolor{grey}{\textbf{Preprocessing the first dimension}}
		\State $Y_{1}= \texttt{reshape}\left(\mathcal{Y}^{k+1}, \left[n_1,n_{2}\dots \newdim^{k+1}\right]\right)$
		\State $R_1^k = \left(I - U^k_{1}U_1^{k^T}\right) Y_1$  \label{alg:line:firstresid} 

		\State $U^{k,pad}_1 \gets U_{1}^{k}$ \Comment{First core has no padding}
		\For{$i=1$ to $d-1$}
		\State \textcolor{grey}{\textbf{Updating $i$-th core $\mathcal{G}_{i}^{k}\rightarrow\mathcal{G}_{i}^{k+1}$}}
		\State $U^k_{R_i}, r_{R_{i}}, V_{R_{i}}^{k} \gets\texttt{SVD}\left(R_i^k, \epsilon_i \right)$ \Comment{$U_{R_i}\in\R{r_{i-1}^{k}n_i\times r_{R_i}}$}
		
		\State $U_i^{k+1}\gets \mat{U^{k,pad}_i& U^{k}_{R_i}}$ \label{alg:line:append}
		\State $U_{i+1}^{k,pad}\gets  \texttt{reshape}\left(\begin{bmatrix}
			\texttt{reshape}\left(U^{k}_{i+1},[r_i^{k}, n_{i+1}r_{i+1}^k]\right) \\
			\textbf{0}_{r_{R_i} \times n_{i+1}{r_{i+1}^k}}
		\end{bmatrix},  [r^{k+1}_{i}n_{i+1}, r_{i+1}^k] \right)$
		\Comment{Pad for rank compat.}
		
		\State \textcolor{grey}{\textbf{Preprocessing the subsequent dimensions of $\mathcal{Y}^{k+1}$}}
		\State $Y_{i+1} \gets \texttt{reshape}\left(U_{i}^{k+1^T}Y_{i}, [r_{i}^{k+1}n_{i+1}, n_{i+2}\cdots \newdim^{k+1}]\right)$ \
		
		\State $R_{i+1}^k\gets \left(I-U_{i+1}^{k,pad}U_{i+1}^{k,pad^T}\right)Y_{i+1}$ 
		
		\EndFor

		\State \textcolor{grey}{\textbf{Updating the $d$-th core}}
		\State $U^k_{R_d}, r_{R_{d}}, V_{R_{d}}^{k} \gets\texttt{SVD}\left(R_d^k, \epsilon_d \right)$ \label{alg:line:truncatedsvd}
		\State $U_d^{k+1}\gets \mat{U^{k,pad}_d& U^{k}_{R_d}}$ 
		\State \textcolor{grey}{\textbf{Updating the last core}}
		\State $Y_{d+1} \gets U_{d}^{k+1^T}Y_{d}$\label{alg:line:lastcore} \Comment{No need to reshape for $Y_{d+1}$ since $U_{d}^{k+1^T}Y_{d}\in\R{r_{d}^{k+1}\times \newdim^{k+1}}$} 
		\State $U_{d+1}^{k+1}\gets \mat{U_{d+1}^{k,pad}& Y_{d+1}}$
		
	\end{algorithmic}
\end{algorithm}

\paragraph{Preprocessing the first dimension of $\mathcal{Y}^{k+1}$}
The first dimension is processed by projecting the reshaped tensor $Y_1 \equiv Y^{k+1}_{(1)}$ onto the orthogonal complement of the space spanned by the first dimension $U_1^k$ through
\begin{equation}\label{eq:first_resid}
	R^k_1=Y^{k+1}_{(1)} - U_{1}^k U_{1}^{k^T}Y^{k+1}_{(1)} = \left(I - U^k_{1}U_1^{k^T}\right)Y_{1}, \quad R_{1}^k \in \reals^{n_1 \times n_2\cdots \newdim^{k+1}},
\end{equation}
as provided in Line~\ref{alg:line:firstresid}.
The column-space of this residual is orthogonal to $U_1^k$, and the corresponding basis is extracted in the updated phase for each dimension via the truncated SVD.

\paragraph{Updating the $i$-th core $\mathcal{G}_i^k \to \mathcal{G}_{i}^{k+1}$}

The update step now uses both the residual $R_{i}^k$ and an intermediate, padded version of the $i$-th core $U_{i}^{k,pad}$ to compute the updated core $U_{i}^{k+1}.$ First the residual is decomposed via the truncated SVD
\begin{equation}
	R^k_{i}=U^{k}_{{R}_{i}}\Sigma^{k}_{{R}_{i}}V^{k}_{{R}_{i}} + E_{i}^k = U^{k}_{{R}_{i}}V^{*}_{{R}_{i}}  + E_{i}^k, \quad \lVert E_i^k \rVert_{F} = \epsilon_{i},
	\label{eq:svdfirst}
\end{equation}
where the second equality multiplies the singular value matrix and the right singular vectors. The left singular vectors are orthogonal to $U_{i}^{k, pad}$, so the updated basis can be obtained by appending these new directions as in Line~\ref{alg:line:append}
\begin{equation}
	U^{k+1}_i=\mat{U^{k,pad}_i &U^{k}_{R_{i}}} \in \R{r^{k+1}_{i-1}n_i\times (r_i^{k+1})},
	\label{eq:coreexpansion}
\end{equation}
where the new TT-rank becomes $r_i^{k+1} = r_{i}^k + r_{R_i}.$ 

The updated TT-core becomes $\mathcal{G}_{i}^{k+1}=\texttt{reshape}\left(U^{k+1}_i, \left[r_{i-1}^{k+1},n_i,r_i^{k+1}\right]\right)$. Note that this update increases the rank between dimensions $i$ and $i+1$ from $r_{i}^k$ to $r_{i}^{k+1}$. This new dimension causes a discrepancy that prohibits the contraction between the $i$-th and $i+1$-th cores. Therefore, the $i+1$-th core must be padded with zeros
\begin{equation}
    U^{k,pad}_{i+1} \leftarrow \texttt{reshape}\left(\begin{bmatrix}
    \texttt{reshape}\left(U^{k}_{i+1},[r_i^{k}, n_{i+1}r_{i+1}^k]\right) \\
    \textbf{0}_{r_{R_i} \times n_{i+1}{r_{i+1}^k}}
    \end{bmatrix},  [r^{k+1}_{i}n_{i+1}, r_{i+1}^k] \right).
    \label{eq:ukpad}
\end{equation}
Note that this padding leaves $U^{k,pad}_{i+1}$ as orthonormal and does not change the span defined by the components of the first $r_{i}^k$ rows of the reshaping.

After this next core, we have a partially updated representation of the accumulation for $\mathcal{\hat{X}}^{k+1}$
\begin{equation}
    \hat{\mathcal{X}}^{k+1} = \mathcal{G}_{1}^{k+1} \  \tenprod{3}{1} \  \cdots \tenprod{3}{1} \  \mathcal{G}_{i}^{k+1} \tenprod{3}{1} \  \mathcal{G}^{k,\textrm{pad}}_{i+1} \  \tenprod{3}{1} \  \mathcal{G}^k_{i+2} \  \tenprod{3}{1} \cdots \tenprod{3}{1} \  \mathcal{G}^{k}_{d+1},
\end{equation}
where $\mathcal{G}^{k, \textrm{pad}}_{i+1}$ refers to the padded core formed by reshaping $U^{k,pad}_{i+1}$ following the update in~\eqref{eq:ukpad}. After this step, we can proceed to pre-process in preparation for updating the next dimension.

\paragraph{Preprocessing the subsequent dimensions of $\mathcal{Y}^{k+1}$} Preprocessing preparation for the update of the next dimension begins by updating the projection of the new data onto the updated core and reshaping 
\begin{equation}
    Y_{i+1} = \texttt{reshape}\left(U_{i}^{k+1^T}  Y_i, [r_{i}^{k+1}n_{i+1}, n_{i+2}\cdots \newdim^{k+1}] \right).
    \label{eq:project_further}
\end{equation}
This projection uses the {\it updated} tensor cores $U_{i}^{k+1}$ and therefore has dimensions corresponding to the updated rank $r_{i}^{k+1}$. Then, the residual with respect to the span of the next core is obtained via a projection onto the orthogonal complement to the existing space
\begin{equation}\label{eq:ithresidual}
	R^k_i= \left(I - U^{k,pad}_{i}U_{i}^{k,pad^T}\right) Y_{i+1}.
\end{equation}

\paragraph{Updating the last core}
The procedure in \eqref{eq:project_further}-\eqref{eq:ukpad}
is repeated sequentially for the first $d$ dimensions. Then after updating $U_{d}^{k,pad}$ to $U_{d}^{k+1}$ and padding $U_{d+1}^{k}$ with zeros as in \eqref{eq:ukpad}, projecting $Y_{d}$ onto $U_{d}^{k+1}$ as in Line~\ref{alg:line:lastcore} yields $Y_{d+1}$, with which we can update $U_{d+1}^{k,pad}$ as

\begin{equation}
	U_{d+1}^{k+1}\leftarrow \begin{bmatrix}
		U_{d+1}^{k,pad} & Y_{d+1}
	\end{bmatrix},
\label{eq:lastcore}
\end{equation}
and obtain the representation for $\hat{\mathcal{X}}^{k+1}$ in TT-format.

\subsection{Analysis}
This section provides proof that TT-ICE can achieve and maintain a target accuracy throughout the compression process.

We first show that Algorithm~\ref{alg:TT-ICE} provides guaranteed reconstruction error of a new tensor given appropriate truncation of the truncated SVDs.
Following the completion of TT-ICE, the approximation of the most recently compressed data point can be obtained by
\begin{equation}
    \hat{\mathcal{Y}}^{k+1} = \mathcal{G}_{1}^{k+1} \  \tenprod{3}{1} \  \cdots \tenprod{3}{1} \  \mathcal{G}^{k+1}_{d} \  \tenprod{3}{1}  \mathcal{G}^{k+1}_{d+1}[-n_{d+1}^{k+1}:], 
    \label{eq:last_core_latent}
\end{equation}
where the index $[-n_{d+1}^{k+1}:]$ denotes the last $n_{d+1}^{k+1}$ columns of the core, as seen in Equation~\eqref{eq:lastcore}.
Recall that $\mathcal{G}^{k+1}_{d+1}$ is a matrix and that the columns $\mathcal{G}^{k+1}_{d+1}[-n_{d+1}^{k+1}:]$ are the only elements of the updated accumulation that are unique to the $(k+1)$-th data point. We now show that the algorithm enables a well-controlled approximation error of this latest tensor.
\begin{theorem}\label{thm:errorbound}
    Let $\{U_{i}^{k}\}_{i=1}^{d+1}$ denote the unfolded TT-cores of the approximate accumulation $\hat{\mathcal{X}}^k$, $\mathcal{Y}^{k+1}$ be a new streamed tensor and $\left\{\epsilon_i\right\}_{i=1}^d$ be the SVD truncation tolerances, then the TT-ICE algorithm computes a TT-approximation $\hat{\mathcal{Y}}^{k+1}$~\eqref{eq:last_core_latent} satisfying
    \begin{equation}\label{eq:errorboundtheorem}
        \lVert \mathcal{Y}^{k+1} - \hat{\mathcal{Y}}^{k+1} \rVert_F \leq \sqrt{\sum_{i=1}^{d} \epsilon_i^2}.
    \end{equation}
\end{theorem}

\begin{proof}

The proof is similar to that of \cite[Thm 2.2]{Oseledets2011} and is by induction. For $d=1$, the statement follows from the properties of the truncated SVD.

Now we consider an arbitrary $d>1$.
The incremental updates start with the first dimension. We start with processing the first core, which works on the first unfolding of the new data. First, we show that computing the SVD of the residual leads to a low-rank approximation of $Y^{k+1}_{(1)}$ with an approximation error bounded by a controlled truncation tolerance. The definitions of the residual and projection given by Equations~\eqref{eq:first_resid} lead to
\begin{align}\label{eq:residualfirstdimension}
    Y^{k+1}_{(1)} &= U_{1}^kU_{1}^{{k}^T}Y^{k+1}_{(1)} + R_{1}^k = U_{1}^k U_{1}^{{k}^T}Y_1 + U_{R_{1}}^{k}V_{R_1}^{k} + E_{1}^{k}  = U_{1}^{k+1} \underbrace{\begin{bmatrix}
     U_{1}^{{k}^T}Y_1 \\
     V_{R_1}^{k} 
     \end{bmatrix}}_{B_1}
     + E_{1}^{k} 
    = U^{k+1}_1 B_1 + E_{1}^{k},
\end{align}
where the third equality appends the existing basis $U_1^k$ and the new directions obtained from the SVD of the residual $U_{R_1}^k$ to form the updated core $U_1^{k+1}.$
Thus, after the first dimension is processed, the first core of $\mathcal{G}_1^{k+1}$ is obtained, as a reshaping of $U_1^{k+1}$. Equation~\eqref{eq:residualfirstdimension} demonstrates that this core is multiplied by a tensor $\mathcal{B}_1\in \reals^{r^{k+1}_{1}n_2 \times n_{3}\cdots n_{d+1}^{k+1}}$ that has reshaping $B_1 \in \reals^{r^{k+1}_1 \times n_{2}\cdots n_{d+1}^{k+1}}$, and that the truncated SVD implies the approximation error $\epsilon_{1}^{2}$:
\begin{align}
     \lVert Y^{k+1}_{(1)}-U_{1}^{k+1}B_1\|^{2}_{F} \leq \epsilon_1^2 .
     \label{eq:proof1}
\end{align}
The algorithm then proceeds to decompose the still high-dimensional $B_1$ into an approximation. Let $\hat{B}_1$ denote this approximation of the remaining dimensions so that we now seek to bound the error
\begin{align}
     \lVert \mathcal{Y}^{k+1} - \hat{\mathcal{Y}}^{k+1} \rVert^{2}_F &= \lVert Y^{k+1}_{(1)}-U_{1}^{k+1}\hat{B}_1\|^{2}_{F}.
\end{align}
Next, this error can be rewritten in terms of the known error~\eqref{eq:proof1} and the subsequent approximation error incurred by $\hat{B}_1$. To this end, add and subtract the exact $B_1$ to obtain
\begin{align}\label{eq:errorbounds1}
     \lVert \mathcal{Y}^{k+1} - \hat{\mathcal{Y}}^{k+1} \rVert_F^2 &= \lVert Y^{k+1}_{(1)}-U_{1}^{k+1}\hat{B}_1\|^{2}_{F} =\lVert Y^{k+1}_{(1)}-U_{1}^{k+1}\left(\hat{B}_1 + B_1 - B_1 \right)  \|^{2}_{F},\\ 
    \label{eq:errorbounds2}&=  \lVert Y^{k+1}_{(1)}-U_{1}^{k+1}B_1\rVert^{2}_{F} + \lVert  U_{1}^{k+1}\left(B_1 - \hat{B}_1\right) \rVert^{2}_{F},\\
    \label{eq:errorbounds3}&\leq \epsilon_1^2 + \lVert  B_1 - \hat{B}_1 \rVert^2_{F},
\end{align}
where $U^{k+1}_{1}$ has orthonormal columns and the inequality in Eq.~\eqref{eq:errorbounds3} arises from Eq.~\eqref{eq:proof1}.

At this stage of the algorithm, we  have updated $U_{1}^{k+1}$ and padded $U_{2}^k$ with $r_{R_{1}}\times n_{2}r_{2}^{k}$ zeros according to~\eqref{eq:ukpad}.  The rest of the algorithm proceeds by now performing the same approximation for $B_1$, beginning with a projection onto the zero padded second core, $U_{2}^{k,pad}.$ First we note that the orthonormality of the left singular vectors implies $B_{1} = U_{1}^{(k+1)^T} Y_{(1)}^{k+1}=U_{1}^{(k+1)^T} Y_{1}$ from \eqref{eq:residualfirstdimension}. Combined with Equation~\eqref{eq:project_further}, these facts imply that $Y_{2} \equiv B_{1}$ up to a reshaping. Then the next residual becomes
\begin{equation}
\left(I - U_{2}^{k,pad} U_{2}^{k,pad^T}\right) B_{1}  \equiv \left(I - U_{2}^{k,pad} U_{2}^{k,pad^T}\right) Y_{2} = R_{2}^{k}  .
\end{equation}
Similar to \eqref{eq:residualfirstdimension}, we can expand $B_1$ into the components parallel to the space spanned by the padded by $U_2^{k,pad}$ and orthogonal to this space so that 
\begin{equation}\label{eq:residualseconddimension}
	B_1=U_{2}^{k,pad}U_{2}^{k,pad^T}B_1+R_{2}^{k} = U_{2}^{k,pad}U_{2}^{k,pad^T}Y_2+U_{R_2}^{k}V_{R_2}^{k}+E_{2}^{k}=U_{2}^{k+1}\underbrace{\begin{bmatrix}
		U_{2}^{k,pad^T}	Y_2 \\V_{R_2}^{k}
	\end{bmatrix}}_{B_2}+E_{2}^{k+1}=U^{k+1}_{2}B_{2}+E_{2}^{k}.
\end{equation}
Similar to \eqref{eq:proof1}, the truncated SVD for this dimension ensured $\lVert E_2^{k} \rVert_{F}^{2} \leq \epsilon_2^2$. Then, parallel to \eqref{eq:errorbounds1}-\eqref{eq:errorbounds3}, this time we expand the error in $B_1$ according to
\begin{equation}
\begin{aligned}
	\|B_{1}-\hat{B}_1 \|_{F}^{2} =& \| B_{1}-U_{2}^{k+1}\hat{B}_{2}\|_{F}^{2},\\
	=&\|B_{1}-U_{2}^{k+1}\left(\hat{B}_{2}+B_{2}-B_{2}\right) \|_{F}^{2}, \\
	=& \|B_{1}-U_{2}^{k+1}B_{2}\|_{F}^{2}+\|U_{2}^{k+1}\left(B_{2}-\hat{B}_{2}\right)\|_{F}^{2},\\
	\leq&\epsilon_{2}^{2}+\|B_{2}-\hat{B}_{2}\|_{F}^{2}.
\end{aligned}
\end{equation}
If we rewrite Equation~\eqref{eq:residualseconddimension} for $B_{d-1}$, we see that 
\begin{equation}\label{eq:residuallastdimension}
	\resizebox*{0.94\hsize}{!}{$
	B_{d-1}=U_{d}^{k,pad}U_{d}^{k,pad^T}B_{d-1}+R_{d}^{k}\equiv U_{d}^{k,pad}U_{d}^{k,pad^T}Y_{d}+U_{R_{d}}^{k}V_{R_{d}}^{k}+E_{d}^{k}=U_{d}^{k+1}\underbrace{\begin{bmatrix}
			U_{d}^{k,pad^T}	Y_{d} \\V_{R_{d}}^{k}
	\end{bmatrix}}_{B_d}+E_{d}^{k+1}=U^{k+1}_{d}B_{d}+E_{d}^{k}
$}
\end{equation}
with $B_{d}\in\R{r_{d}^{k+1}\times n_{d+1}^{k+1}}$, which is equal to $Y_{d+1}$. We directly append $Y_{d+1}$ to the $d+1$-th core and conclude our update procedure for $\mathcal{Y}^{k+1}$. Therefore, for a $d+1$ dimensional tensor, repeating the update for the first $d$ dimensions successively yields the approximation error

\begin{equation}\label{eq:errorsummation}
	\lVert \mathcal{Y}^{k+1} - \hat{\mathcal{Y}}^{k+1}\rVert_{F}^{2}\leq \sum_{i=1}^{d} \epsilon_i^2.
\end{equation}
Computing the square root of this term yields \eqref{eq:errorboundtheorem} and thus concludes our proof.
\end{proof}
A straightforward corollary shows that TT-ICE can then ensure that the compression of the new data point can be achieved to any desired tolerance. 
\begin{corollary}\label{cor:relativeerror}
    Let $\{U_{i}^{k}\}_{i=1}^{d+1}$ denote the unfolded TT-cores of the approximate accumulation $\hat{\mathcal{X}}^k$, $\mathcal{Y}^{k+1}$ be a new streamed tensor. If $\epsilon_{i} = \epsilon=\varepsilon_{des}\|\mathcal{Y}^{k+1}\|_{F}/\sqrt{d}$,
    then 
    \begin{equation}
    \frac{\|\mathcal{Y}^{k+1}-\hat{\mathcal{Y}}^{k+1}\|_{F}}{\|\mathcal{Y}^{k+1}\|_{F}}\leq\varepsilon_{des}
    \end{equation}
    for any $\varepsilon_{des}>0.$
    
\end{corollary}
\begin{proof}
    The proof is a direct result of Theorem~\ref{thm:errorbound}. If we assume $\epsilon_i=\epsilon$ for all $i$ and simply set
	\begin{equation}
	    \sqrt{\sum_{i=1}^{d}\epsilon_i^2}=\epsilon\sqrt{d}=\varepsilon_{des}\|\mathcal{Y}^{k+1}\|_{F},
	\end{equation}
we can assure the desired relative error upper bound $\varepsilon_{des}$ by truncating the SVD on Line~\ref{alg:line:truncatedsvd} of TT-ICE at $\epsilon=\varepsilon_{des}\|\mathcal{Y}^{k+1}\|_{F}/\sqrt{d}$.
\end{proof}
The last step is to ensure that the update of $\mathcal{X}^k$ does not reduce the approximation error of the previously compressed elements $\mathcal{Y}^{i}$ for $i \leq k$.

\begin{theorem}\label{thm:maintainerrorbound}
    Let $\{U_{i}^{k}\}_{i=1}^{d+1}$ denote the unfolded TT-cores of the approximate accumulation $\hat{\mathcal{X}}^k$ such that the approximation of any existing tensor $\mathcal{Y}^{\ell}$ satisfies $\lVert \mathcal{Y}^{\ell}-\hat{\mathcal{Y}}^{{\ell}}\rVert_{F} \leq \epsilon \lVert\mathcal{Y}^{\ell}\rVert_F$ for $\ell=1,\ldots,k$, and $\epsilon > 0$. Let $\mathcal{Y}^{k+1}$ be a new streamed tensor, and $\left\{\epsilon_i\right\}_{i=1}^d$ be the SVD truncation tolerances; then the updated TT-cores $\{U_{i}^{k+1}\}_{i=1}^{d+1}$ computed by TT-ICE represent an approximate accumulation $\hat{\mathcal{X}}^{k+1}$ that still satisfies $\lVert \mathcal{Y}^{\ell}-\hat{\mathcal{Y}}^{{\ell}}\rVert_{F} \leq \epsilon \lVert\mathcal{Y}^{\ell}\rVert_F.$ for $\ell = 1,\ldots,k.$
\end{theorem}
\begin{proof}

    The proof shows simply that the core modifications performed by TT-ICE have no impact on the representation of earlier tensors. First recall from~\eqref{eq:last_core_latent}
    that the approximation of the $\ell$-th tensor is given by
    \begin{equation}
        \hat{\mathcal{Y}}^{\ell} = \mathcal{G}_{1}^{k} \  \tenprod{3}{1} \  \cdots \tenprod{3}{1} \  \mathcal{G}^{k}_{d} \  \tenprod{3}{1}  \mathcal{G}^{k}_{d+1}\left(\mathcal{Y}^{\ell}\right), \quad \ell = 1,\ldots,k,
        \label{eq:temp1}
    \end{equation}
    where we slightly abuse the notation by letting $\mathcal{G}^{k}_{d+1}\left(\mathcal{Y}^{\ell}\right)$ denote the columns of $\mathcal{G}^{k}_{d+1}$ corresponding to $\mathcal{Y}^{\ell}$.

    First, we note that TT-ICE only appends columns to the last core, so that the columns corresponding to the $\ell$-th tensor are unchanged when $\mathcal{G}^{k}_{d+1}$ is updated to $\mathcal{G}^{k+1}_{d+1}$. Thus it remains to show that the portions of $\mathcal{G}_{i}^k$ that multiply together and finally multiply the corresponding columns of $\mathcal{G}_{d+1}^k$ remain unchanged as the $U_{i}^k$ are updated to $U_{i}^{k+1}.$ 

    Let $G_{i}^k[i_1] \in \reals^{r_{i-1} \times r_{i}}$ denote the slice of each TT-core for $i=1\ldots n_{i}.$  Then~\eqref{eq:temp1} can be written for each element of $\mathcal{Y}^\ell$ according to
     \begin{equation}
		\hat{\mathcal{Y}}^{\ell}(i_1,\dots,i_d, :\mathcal{Y}^\ell)=\mathcal{G}_{1}^{k}[i_1]\mathcal{G}_{2}^{k}[i_2]\cdots\mathcal{G}_{d}^{k}[i_d]\mathcal{G}_{d+1}^{k}[:\mathcal{Y}^\ell],
	\end{equation}
 where $[:\mathcal{Y}^{\ell}]$ is used to denote the columns of the last core pertaining to $\hat{Y}^{\ell}$ so that $\mathcal{G}_{d+1}^{k}[:\mathcal{Y}^\ell] \equiv \mathcal{G}^{k}_{d+1}\left(\mathcal{Y}^{\ell}\right).$ Next, the zero padding and column appending of each $U_{i}^k$ imply the following relationship between $G_{i}^k$ and $G_{i}^{k+1}$
    \begin{equation}
	   \mathcal{G}^{k+1}_{j}[i_j]=
    	\begin{cases}
    	   	\mat{\mathcal{G}_{1}^{k}[i_1]& L_{1}} &,\ j=1;\\ \\
    		\mat{\mathcal{G}_{j}^{k}[i_j]&L_{j,1}\\0&L_{j,2}}& ,\ j=2,\dots d;\\ \\
    		\mat{\mathcal{G}_{d+1}^{k}[i_{d+1}]\\0}&,\ j=d+1,
    	\end{cases}
    	\label{eq:updatedslices}
    \end{equation}
    where the zeros in the lower left block come from the padding, and the 
    $L_{j,l}$ matrices arise from appending the new directions.
    Now the representation of the $\ell$-th tensor in the stream with the new cores becomes
    \begin{equation}\label{eq:updateequalitydcores}
    	\hat{\mathcal{Y}}^{\ell}=\mat{\mathcal{G}_{1}^{k}[i_1]& L_{1}}\mat{\mathcal{G}_{2}^{k}[i_2]&L_{2,1}\\0&L_{2,2}}\cdots\mat{\mathcal{G}_{d}^{k}[i_{d}]&L_{d,1}\\0&L_{d,2}}\mat{\mathcal{G}_{d}^{k}[i_{d+1}]\\0}=\mathcal{G}_{1}^{k}[i_1]\mathcal{G}_{2}^{k}[i_2]\cdots\mathcal{G}_{d}^{k}[i_{d+1}]=\hat{\mathcal{Y}}^{\ell},
    \end{equation}
    where the locations of the zeros enable the new terms to appropriately cancel all the new terms. Thus the representation of previously compressed tensors does not change, and therefore their reconstruction error remains the same after the core update.
\end{proof}

\subsection{Heuristics to improve performance}\label{sec:heuristics}
This section discusses three heuristic modifications to TT-ICE to reduce the computational load and memory requirements of our approach. 

These approaches aim to reduce the number of core updates between $\mathcal{X}^k$ and $\mathcal{X}^{k+1}.$ First, we provide a metric based on TT-ranks to decide whether an attempt at updating a TT-core should even be made. Second, when a new batch of $\newdim^{k+1}$ tensors is presented at increment $k+1$, we describe an approach to subselect from this batch to only perform an update with a smaller number of tensors. This update is based on each of the approximation errors of individual tensors in the streamed batch. Third, we propose a mechanism to determine when it is feasible to skip the core updating process.

\paragraph{Core occupancy for reducing the number of cores to be updated}
The TT-ICE Algorithm~\ref{alg:TT-ICE} proposes a scheme for updating all the TT-cores $\{\mathcal{G}^{k}_{i}\}_{i=1}^{d+1}$ at each increment. However, the new information provided by increment may not be uniform throughout all these TT-cores. If we can detect the core where the most information loss occurs, focusing update efforts on that core could increase the efficiency of Algorithm~\ref{alg:TT-ICE}.

We propose a heuristic named \textit{core occupancy} to measure how much information is already represented by a core to determine if it should be updated. Recall that the $i$-th TT-core is constructed by reshaping the left singular vectors $U_{i}^{k}\in\R{r^{k}_{i-1}n_{i}\times r_{i}^{k}}$. As a result, there are $r^{k}_{i}$ orthogonal vectors from the possible $r^{k}_{i-1}n_{i}$ basis vectors. For a given TT-core, core occupancy represents the ratio of the truncation rank $r^{k}_{i}$ over the maximum rank possible (the number of rows $r^{k}_{i-1}n_{i})$. It can also be seen as a ratio of the column to row-ranks:
\begin{equation}\label{eq:coreoccupancy}
\texttt{occupancy}\left(\mathcal{G}^{k}_{i}\right)=\frac{r^{k}_{i}}{r^{k}_{i-1}n_{i}}.
\end{equation}
If a core has a high occupancy ratio, then we can skip its update.

\paragraph{Subselecting the observations of $\mathcal{Y}^{k+1}$ used to update $\{\mathcal{G}^{k}_{i}\}_{i=1}^{d+1}$}
If the new data comes in a batch (i.e. $\newdim^{k+1}>1$), another way of reducing the computational cost of updating the TT-cores is to use a reduced number of observations from the new tensor $\mathcal{Y}^{k+1}$ to update the first $d$ cores of the accumulation tensor. To subselect which elements of the batch to use for the update, we first compute the projection of $\mathcal{Y}^{k+1}$ onto TT-cores $\{\mathcal{G}^{k}_{i}\}_{i=1}^{d+1}$ and use the relative projection error of the individual observations as a heuristic to select a subset of $\mathcal{Y}^{k+1}$. With a slight abuse of notation, we denote the $i$-th observation in $\mathcal{Y}^{k+1}$ as $\mathcal{Y}^{k+1}(i)$, where $i=1,\ldots, \newdim^{k+1}.$ Let $\tilde{\mathcal{Y}}^{k+1}$ be the approximation of $\mathcal{Y}^{k+1}$ using the TT-cores $\{\mathcal{G}^{k}_{i}\}_{i=1}^{d+1}$. Similarly, let $\tilde{\mathcal{Y}}^{k+1}(i)$ denote the $i$-th observation in $\tilde{\mathcal{Y}}^{k+1}$. Then, the vector $\varepsilon_{\mathcal{Y}^{k+1}}\in\R{\newdim^{k+1}}$ is vector of relative errors of individual observations, where
\begin{equation}\label{eq:relerror}
	\varepsilon_{\mathcal{Y}^{k+1}}(i)=\frac{\| \mathcal{Y}^{k+1}(i)-\tilde{\mathcal{Y}}^{k+1}(i) \|_F}{\|\mathcal{Y}^{k+1}(i) \|_F},
\end{equation}
for $i=1,\dots,\newdim^{k+1}$.

Recall that we can use Algorithm~\ref{alg:TT-ICE} with truncated SVD and update the TT-cores of $\hat{\mathcal{X}}^k$ to $\hat{\mathcal{X}}^{k+1}$ so that $\mathcal{Y}^{k+1}$ is represented within a truncation error threshold. Let $\varepsilon_{des}$ denote the determined relative error threshold as in Corollary~\ref{cor:relativeerror}. Before attempting an update, we calculate the approximation error of each new data point $\varepsilon_{\mathcal{Y}^{k+1}}$ with the existing TT-cores using \eqref{eq:relerror} and then compute its average, $\texttt{mean}\left(\varepsilon_{\mathcal{Y}^{k+1}}\right).$ If the average approximation error is greater than the desired error tolerance ($\texttt{mean}\left(\varepsilon_{\mathcal{Y}^{k+1}}\right)>\varepsilon_{des}$), then we will proceed to update at least some of the cores.
In particular, we only update the TT-cores of the accumulation using the data points for which the approximation error exceeds $\varepsilon_{des}$. Let $\mathcal{D}^{k+1}$ be that constructed new $(d+1)$-way tensor with a reduced number of observations such that
\begin{equation}\label{eq:subselectheuristic}
	\mathcal{D}^{k+1}=\left\{\mathcal{Y}^{k+1}(i):\ \varepsilon_{\mathcal{Y}^{k+1}}(i)>\varepsilon_{des}\right\}_{i=1}^{\newdim^{k+1}}.
\end{equation}

Since the algorithm operates on a subset of observations, the truncation parameter needs to be adjusted accordingly. Using the same $\varepsilon_{des}$ as $\mathcal{Y}^{k+1}$ for $\mathcal{D}^{k+1}$ enforces a tighter relative error upper bound on $\mathcal{Y}^{k+1}$. This tighter error bound results in increased TT-ranks that impair the compression performance. To avoid this, we need to relax the truncation parameter $\varepsilon_{des}$ using the approximation error of the discarded observations. Let $\mathcal{D}_{C}^{k+1}$ be the tensor constructed with those discarded observations such that $\mathcal{D}_{C}^{k+1}=\left\{\mathcal{Y}^{k+1}(i):\ \varepsilon_{\mathcal{Y}^{k+1}}(i)\leq\varepsilon_{des}\right\}_{i=1}^{\newdim^{k+1}}$. Then, we compute the relaxed relative error tolerance $\varepsilon_{upd}$ for the tensor of selected observations $\mathcal{D}^{k+1}$ as 
\begin{equation}\label{eq:epsilonmodification}
	\varepsilon_{upd} = \sqrt{ \frac{\left(\varepsilon_{des}\|\mathcal{Y}^{k+1}\|_{F}\right)^{2}-\|\mathcal{D}^{k+1}_{C}-\tilde{\mathcal{D}}^{k+1}_{C}\|_{F}^{2}}{\|\mathcal{D}^{k+1}\|_{F}^{2}} },
\end{equation}
where $\tilde{\mathcal{D}}^{k+1}_{C}$ is the approximation of $\mathcal{D}_{C}^{k+1}$ using $\{U^{k}_{i}\}_{i=1}^{d+1}$. The first term in the numerator is the maximum amount of error that TT-ICE can allow for $\mathcal{Y}^{k+1}$ in the Frobenius norm, the second term is the approximation error of the discarded observations in the Frobenius norm, and the denominator is the Frobenius norm of the tensor of selected observations. Since the discarded observations have a relative error $\leq \varepsilon_{des}$, subtracting the approximation error gives how much error TT-ICE can allow if it only uses $\mathcal{D}^{k+1}$ to update the TT-cores. The denominator functions as a normalizing factor, converting the error in the numerator to a relative error for incremental updates. Through $\varepsilon_{upd}$, the low approximation error of the discarded observations will be balanced by tolerating slightly more error for $\mathcal{D}^{k+1}$ and the overall approximation $\hat{\mathcal{Y}}^{k+1}$ will have a relative error closer to $\varepsilon_{des}$. Note that for the edge case where all $\varepsilon_{\mathcal{Y}^{k+1}}(i)>\varepsilon_{des}$, we have $\mathcal{D}^{k+1}=\mathcal{Y}^{k+1}$ and $\mathcal{D}^{k+1}_{C}=\emptyset$. This results in $\varepsilon_{upd}=\varepsilon_{des}$, therefore providing a consistent method to update $\varepsilon_{des}$.

If the observations in the same batch have similar norms, we can approximate Eq.~\eqref{eq:epsilonmodification} by replacing the norm operators with the count of observations in each tensor $|\cdot|$ and get
\begin{equation}\label{eq:epsilonmodification2}
	\varepsilon_{upd}\approx\frac{\varepsilon_{des}\newdim^{k+1}-\varepsilon_{\mathcal{D}_C^{k+1}}|\mathcal{D}_{C}^{k+1}|}{|\mathcal{D}^{k+1}|},
\end{equation}
 where $\varepsilon_{\mathcal{D}_C^{k+1}}$ is the mean relative error of the observations in $\mathcal{D}_{C}^{k+1}$ computed analogously to \eqref{eq:relerror}.
 
 Through these modifications, we prevent a superfluous increase in TT-ranks. After determining $\varepsilon_{upd}$, we simply compute the truncation parameter for the SVD using $\varepsilon_{upd}$ as $\epsilon=\frac{\varepsilon_{upd}}{\sqrt{d}}\|\mathcal{D}^{k+1}\|_F$.

\paragraph{Skip updating the first $d$ TT-cores}
Our final heuristic is to skip updating the cores if the average error of the tensors in a batch is less than the desired threshold: $\texttt{mean}\left(\varepsilon_{\mathcal{Y}^{k+1}}\right)\leq\varepsilon_{des}$. 
We justify this heuristic via the following argument.

	Let $\varepsilon_{k+1}$ represent the relative error of approximating the new tensor via the existing cores
	\begin{equation}\label{eq:batchrelativeerror}
		\varepsilon_{k+1}=\frac{\|\mathcal{Y}^{k+1}-\tilde{\mathcal{Y}}^{k+1}\|_{F}}{\|\mathcal{Y}^{k+1}\|_{F}}=\sqrt{\frac{\sum_{i=1}^{n^{k+1}_{d+1}}\|\mathcal{Y}^{k+1}(i)-\tilde{\mathcal{Y}}^{k+1}(i)\|_{F}^{2}}{\sum_{i=1}^{n^{k+1}_{d+1}}\|\mathcal{Y}^{k+1}(i)\|_{F}^{2}}},
	\end{equation}
	where $\tilde{\mathcal{Y}}^{k+1}$ is the approximation of $\mathcal{Y}^{k+1}$ with the existing TT-cores. If $\varepsilon_{k+1}\leq \varepsilon_{des}$, then the first $d$ TT-cores of $\mathcal{X}_k$ can represent $\mathcal{Y}^{k+1}$ sufficiently accurately and do not need updates to meet the desired accuracy. As a result, the error-truncated SVD in Line 13 of \Cref{alg:TT-ICE} returns empty matrices and TT-ICE will return the first $d$ TT-cores without an update. In order to save invaluable computation time, we can calculate $\varepsilon_{k+1}$ using \eqref{eq:batchrelativeerror} and complete the core update by appending the TT-representation of $\mathcal{Y}^{k+1}$ after projecting onto the TT-cores of $\mathcal{X}^{k}$.
	
	Let $\hat{Y}^{k+1}$ be the projection of $\mathcal{Y}^{k+1}$ onto the first $d$ TT-cores $\{\mathcal{G}^{k}_{i}\}_{i=1}^{d}$. Then, updating $\mathcal{G}^{k}_{d+1}$ is simply done by appending $\hat{Y}^{k+1}$ to $\mathcal{G}_{d+1}^{k}$ as
	\begin{equation}\label{eq:appendupdate}
	 	\mathcal{G}_{d+1}^{k+1}\gets\mat{\mathcal{G}_{d+1}^{k}&\hat{Y}^{k+1}}.
	\end{equation}

	However, an explicit computation of \eqref{eq:batchrelativeerror} can become computationally expensive if the batch consists of a high number of observations. In that case, we investigate the use of the surrogate of $\varepsilon_{\mathcal{Y}^{k+1}}$ from \eqref{eq:relerror} as an approximation to $\varepsilon_{k+1}$. In future works, this strategy might be adapted to sampling tensors in the batch to have a stochastic approximation of the average. Meanwhile, the numerical experiments in~\cref{sec:experiments} use $\texttt{mean}(\varepsilon_{\mathcal{Y}^{k+1}})$ to determine if TT-ICE${}^{*}$ should skip updating the first $d$ TT-cores. As an example, \Cref{fig:mspacmanerror} provides empirical proof that using the mean as an approximation does not result in a violation of the relative error upper bound. However, note that there might be other, more conservative heuristic measures available to approximate $\varepsilon_{k+1}$.

The pseudocode of the modified algorithm is provided in Algorithm~\ref{alg:coreupdateheuristic}. The \texttt{project} function on Line~\ref{alg:line:projectfunction} projects $\mathcal{Y}^{k+1}$ sequentially onto $\{U^{k+1}_{i}\}_{i=1}^{d}$ and reshapes into proper dimensions in a way similar to Line~\ref{alg:line:project}.

\begin{algorithm}[h!]
	\caption{\texttt{TT-ICE$^*$}: Incremental update of a tensor train decomposition with heuristic performance upgrades}
	\begin{algorithmic}[1]
	    \Input 
		\Desc{$\left\{U^{k}_i\right\}_{i=1}^{d+1}$}{$\quad$ reshaped cores of the TT-decomposition of the accumulation $\mathcal{X}^k$} 
		\Desc{$\mathcal{Y}\in\R{n_1\times\cdots\times n_d\times \newdim^{k+1}}$}{$\quad$ new tensor}
		\Desc{$\tau$}{$\quad$ occupancy threshold (suggested=0.8)}
		\Desc{$\varepsilon_{des}$}{$\quad$ relative error upper bound}
		\EndInput
		\Output 
		\Desc{$\left\{U^{k+1}_i\right\}_{i=1}^{d+1}$}{$\quad$ updated cores for the accumulation $\mathcal{X}^{k+1}$} 
		\EndOutput
		\State \textcolor{grey}{\textbf{Check representation accuracy with existing TT-cores and \textit{Skip} updating if sufficient}}
		\State $\varepsilon_{\mathcal{Y}}(i)\gets\left.\frac{\|\mathcal{Y}(i)-\tilde{\mathcal{Y}}(i)\|_{F}}{\|\mathcal{Y}(i)\|_{F}}\right|_{i=1}^{\newdim^{k+1}}$\Comment{$\tilde{\mathcal{Y}}$ is the approximation of $\mathcal{Y}$ using $U^{k}$, and $\varepsilon_{\mathcal{Y}}\in\reals^{\newdim^{k+1}}$ as defined in \eqref{eq:relerror}}
		\If{\texttt{mean}$(\varepsilon_{\mathcal{Y}})\leq\varepsilon_{des}$}
		\State $\{U^{k+1}_{i}\}_{i=1}^{d}\gets\{U^{k}_{i}\}_{i=1}^{d}$
		\Else
		\State \textcolor{grey}{\textbf{\textit{Subselect} observations}}
		\State $\mathcal{D}\gets\left\{\mathcal{Y}(j):\ \varepsilon_{\mathcal{Y}}(j)>\varepsilon_{des}\right\}_{j=1}^{\newdim^{k+1}}$ \Comment{$\mathcal{D}\subset\mathcal{Y}$ will be used at update}
		\State $	\varepsilon_{upd} \gets \sqrt{ \frac{\left(\varepsilon_{des}\|\mathcal{Y}\|_{F}\right)^{2}-\left(\|\mathcal{D}_{C}-\tilde{\mathcal{D}}_{C}\|_{F}\right)^{2}}{\|\mathcal{D}\|_{F}^{2}} }$ \Comment{$\mathcal{D}\cup\mathcal{D}_{C}=\mathcal{Y}$, $\tilde{\mathcal{D}}_{C}$ is approximation of $\mathcal{D}_{C}$ using $U^{k}$}
		\State $\epsilon=\frac{\varepsilon_{upd}}{\sqrt{d}}\|\mathcal{D}\|_F$ \Comment{$\epsilon$ is the truncation parameter for SVD}
		\State \textcolor{grey}{\textbf{Perform incremental updates with the selected observations}}
		\State $D_{1}=\texttt{reshape}(\mathcal{D},[n_{1},n_{2}\dots n_{d}\newdim^{\mathcal{D}}])$ \Comment{$\newdim^{\mathcal{D}}$ is the number of selected observations}
		\State $U_{1}^{k,pad}\gets U_{1}^{k}$ \Comment{First core has no padding}
		
		\For{$i=1$ to $d$}
        \State \textcolor{grey}{\textbf{Check core \textit{Occupancy}}}
		\If{\texttt{occupancy}$(U_{i}^{k,pad})\geq\tau$}
       \label{alg:line:loopstart}
		\State $U_{i}^{k+1}\gets U_{i}^{k,pad}$
		\State $r_{i}^{k+1}\gets r_{i}^{k}$
		\State $U_{i+1}^{k,pad}\gets U_{i+1}^{k}$
		\Else
        \State \textcolor{grey}{\textbf{Perform incremental update}}
		\State $R^{k}_{i}=\left(I-U_{i}^{k,pad}U_{i}^{k,pad^T}\right)D_{i}$
		\State $U_{R_i}^{k}\gets \texttt{SVD}\left(R_{i}^{k},\epsilon\right)$
		\State $U_i^{k+1}\gets \mat{U_{i}^{k,pad}&U_{R_i}^{k}}$
		\State $U_{i+1}^{k,pad}\gets \texttt{reshape}\left(\mat{\texttt{reshape}\left(U_{i+1}^{k} ,\left[ r_{i}^{k}, n_{i+1}r_{i+1}^{k}\right]\right)\\\mathbf{0}_{r_{R_i}\times n_{i+1}r_{i+1}^{k}}},[r_{i}^{k+1}n_{i+1},r_{i+1}^{k}]\right)$
		\EndIf
		\State $D_{i+1}\gets \texttt{reshape}(U_{i}^{k+1^T}D_{i},[r_{i}^{k+1}n_{i+1},n_{i+2}\dots \newdim^{\mathcal{D}}])$ \label{alg:line:project}
		\EndFor
		\EndIf
		\State $\hat{Y}^{k+1}\gets\texttt{project}(\mathcal{Y}^{k+1},\{U^{k+1}_{i}\}_{i=1}^{d})$ \Comment{$\hat{Y}$ is obtained by sequentially projecting $\mathcal{Y}^{k+1}$ onto $\{U^{k+1}_{i}\}_{i=1}^{d}$} \label{alg:line:projectfunction}
		\State $U^{k+1}_{d+1}\gets \mat{U^{k}_{d+1}&\hat{Y}^{k+1}}$ \Comment{$\hat{Y}^{k+1}\in\reals^{r_{d}\times n_{d+1}^{k+1}}$ are the columns }
	\end{algorithmic}
	\label{alg:coreupdateheuristic}
\end{algorithm}

\section{Experiments}
\label{sec:experiments}

In this section, we compare the performance of TT-ICE and TT-ICE$^*$ with existing approaches in compressing large-scale data including videos from gameplay sequences and grid data from physics-based simulations. 
In all the experiments, the TT-SVD algorithm~\cite[Alg 1]{Oseledets2011} is applied to obtain an initial set of TT-cores and subsequently, the cores are updated using each of the incremental algorithms. In all of the results provided, $\text{ITTD}k$ corresponds to updating the TT-cores of the accumulation using ITTD and performing TT-rounding~\cite[Alg 2]{Oseledets2011} after every $k$-th update step. 
For comparisons between TT-FOA and our approach, we refer to~\Cref{app:ttfoacomparison}. We do not include it here because TT-FOA does not support rank adaptation and does not offer an upper bound on approximation error. As a result, it either performed poorly or was unable to handle the datasets considered.

\subsection{Evaluation criteria}\label{sec:evaluationcritera}
This section contains performance metrics we use to compare performance of algorithms.

The performance of the incremental algorithms is evaluated using three main criteria. First, the {\em compression ratio} (CR) of the accumulation $\mathcal{X}^{k}$ with core sizes $n_1\times n_2\times\cdots\times n_{d+1}$ is defined as 
\begin{equation}\label{eq:compressionratio}
	CR=\frac{\text{number of elements of full tensor}}{\text{number of parameters in compressed representation}} = \frac{\prod_{i=1}^{d+1}n_i}{\sum_{i=1}^{d+1}r_{i-1}n_{i}r_{i}},
\end{equation}
where $r_i$ is the $i$-th TT-rank. Second, the {\em relative reconstruction error} (RRE) of the same tensor is given by
\begin{equation}\label{eq:MRRE}
	RRE=\frac{\|\mathcal{X}^{k}-\hat{\mathcal{X}}^{k}\|_{F}}{\|\mathcal{X}^{k}\|_{F}}.
\end{equation}
Similarly, we can measure the error in the representation of {\em unseen} data: since the last dimension of $\mathcal{X}^{k}$ provides a latent representation for individual observations stored in the stream, as shown in Eq.~\eqref{eq:last_core_latent}, we can use the first $d$ cores to estimate a latent representation for unseen data by projecting onto the first $d$ dimensions.
We shall call the error in the estimation the {\em relative prediction error} (RPE), which can again be computed using the Frobenius error similar to RRE.
An RPE lower than the target tolerance $\varepsilon_{des}$ indicates that the existing basis of the accumulation is expressive enough to represent this unseen data. Our final evaluation metric is the {\em execution time}, which measures the CPU time spent executing the steps of the corresponding incremental algorithms and does not include the time for data to load into memory.

\subsection{Datasets and experiments}

This section includes tests on two different types of  datasets (snapshots in Figure~\ref{fig:examplefigures}): (i) Raw pixel data from an ATARI game-playing reinforcement learning (RL) agent \cite{chen2021behavioral} and (ii) Grid data from a numerical simulation of self-oscillating gels~\cite{Alben2019}.

\begin{figure}[h!]
\begin{subfigure}{0.48\textwidth}
	\centering
	\includegraphics[width=0.5\textwidth]{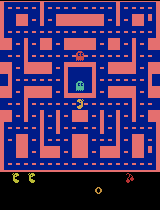}
	\caption{}
	\label{fig:mspacmanframe}
\end{subfigure}
\hfill
\begin{subfigure}{0.48\textwidth}
	\centering
	\includegraphics[width=0.8\textwidth]{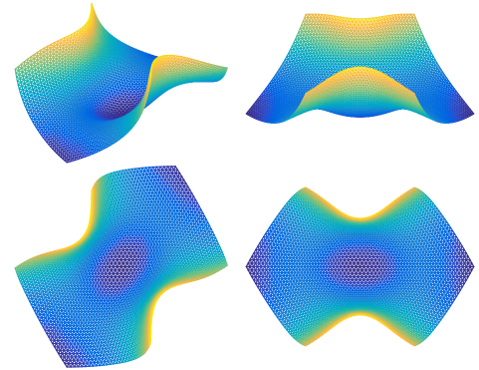}
	\caption{}
	\label{fig:examplegels}
\end{subfigure}
\mcaption{Example visualizations from the datasets used in the compression experiments. (a) A frame from a Ms.Pac-Man gameplay session. (b) Snapshots from self-oscillating gel simulation using the simulation developed in \cite{Alben2019}.}
\label{fig:examplefigures}
\end{figure}

\subsubsection{ATARI gameplay sequences}
This section includes tests of TT-ICE's performance to compress raw pixel data.

The first dataset involves a sequential set of video game screenshots. It is often useful to reduce the dimension of this type of video data to enable learning in the latent space. This dataset consists of a collection of gameplay screen captures from various ATARI games. The dataset is collected using a video game-playing RL agent trained by a Deep Q-Network model \cite{Mnih2015} from the RL-Baselines Zoo package \cite{rl-zoo} and further trained using the Stable Baselines platform \cite{stable-baselines}. The ATARI games we used are Ms.Pac-Man, Enduro, Seaquest, Q*bert, Breakout, Pong, and Beamrider. We present the results of experiments for Ms.Pac-Man game captures in this section and the rest of the games in the Appendix.

Each game has multiple individual gameplay sessions, which we refer to as \textit{runs}. Each run may have different durations and therefore a different number of frames (see Figure~\ref{fig:mspacmanframe}) with dimensions $210\times160\times3$ ($Width\times Height \times RGB$). Each individual run is treated as an incremental unit and reshaped into a 5-way tensor of dimension $30\times28\times40\times3\times n_{5}^{k}$, where $n_{5}^{k}$ is the number of frames in the $k$-th run. Depending on the duration of each gameplay sequence, each run consists of a varying number of frames. Therefore, the 5-way tensors are stacked along the fifth dimension.

We compared TT-ICE, TT-ICE$^*$, and ITTD5 on all of the different game datasets, and the results are summarized in Table~\ref{tab:atariresults}. A more detailed study of the efficiency of each algorithm was done for Ms.Pac-Man. For the comprehensive experiments with Ms.Pac-Man gameplay sequences, we used a training set of 60 runs for incremental updates and a validation set of 160 runs to measure the prediction error on unseen runs. This prediction tests the performance of the TT-format to find a suitable latent space for describing Ms.Pac-Man frames. We present the results of detailed experiments in \Cref{fig:mspacmanerror,fig:mspacmancompression,fig:mspacmantime}. In those figures, \textit{Subselect} means updating TT-cores using only the observations from $\mathcal{Y}^{k+1}$ with relative error higher than $\varepsilon_{des}$ and \textit{SS} means skipping updates for the first $d$ TT-cores when $\texttt{mean}(\varepsilon_{\mathcal{Y}^{k+1}})\leq \varepsilon_{des}$ along with \textit{Subselect} heuristic.
Please note that the \textit{Occupancy} heuristic is only used when the complete TT-ICE$^*$ algorithm is used.

For all other games in Table~\ref{tab:atariresults}, we used a training set of 60 runs and a validation set of 100 runs. Enduro and Pong were the only exceptions to that setup since they had a much higher number of frames in each run. This resulted in much higher memory requirements than all other games even just for storing the runs. Again, due to the high number of frames per run, the validation tests required a significant amount of time. Therefore we used a training set of 40 runs and a validation set of 40 runs. During the repetitions for different algorithms, the streaming order was the same for all methods subject to investigation. All experiments were repeated for two $\varepsilon_{des}$ settings, $\varepsilon_{des}=0.1$ and $\varepsilon_{des}=0.01$. Finally, an occupancy threshold of 0.8 is set for this class of experiments.

\paragraph{\textbf{Compression results}}

This section includes the results of experiments on Ms.Pac-Man frames, and then comment on the compression experiments with other ATARI games.

\Cref{fig:mspacman01time,fig:mspacman001time} show the influence of each variation of TT-ICE on the overall compression time and compare these results with ITTD5. In Figure~\ref{fig:mspacman01time}, ITTD5 performs worse than all of the TT-ICE variations at $\varepsilon_{des}=0.1$. Specifically, Figure~\ref{fig:mspacman01time} depicts that implementing heuristic improvements provides at least $62\%$ reduction in compression time. This improvement in compression time reaches its peak when all of the heuristic upgrades are implemented (i.e. when the complete TT-ICE$^*$ algorithm is used). Note that when the truncation tolerance is tightened to $\varepsilon_{des}=0.01$, the problem with ITTD intensifies. This time, ITTD5 fails to compress the full duration of the stream due to insufficient memory and fails at the $10^{th}$ increment while attempting TT-rounding. On the other hand, TT-ICE and variations of TT-ICE$^*$ display performances in parallel to the experiments with $\varepsilon_{des}=0.1$.

\begin{figure}[t]
	\begin{subfigure}{\textwidth}
		\centering
		\includegraphics[width=0.5\textwidth]{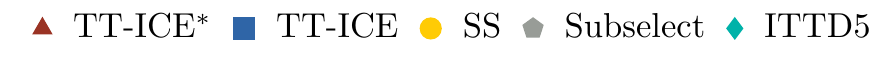}
	\end{subfigure}

\begin{subfigure}{0.49\textwidth}
	\centering
	\includegraphics[width=\textwidth]{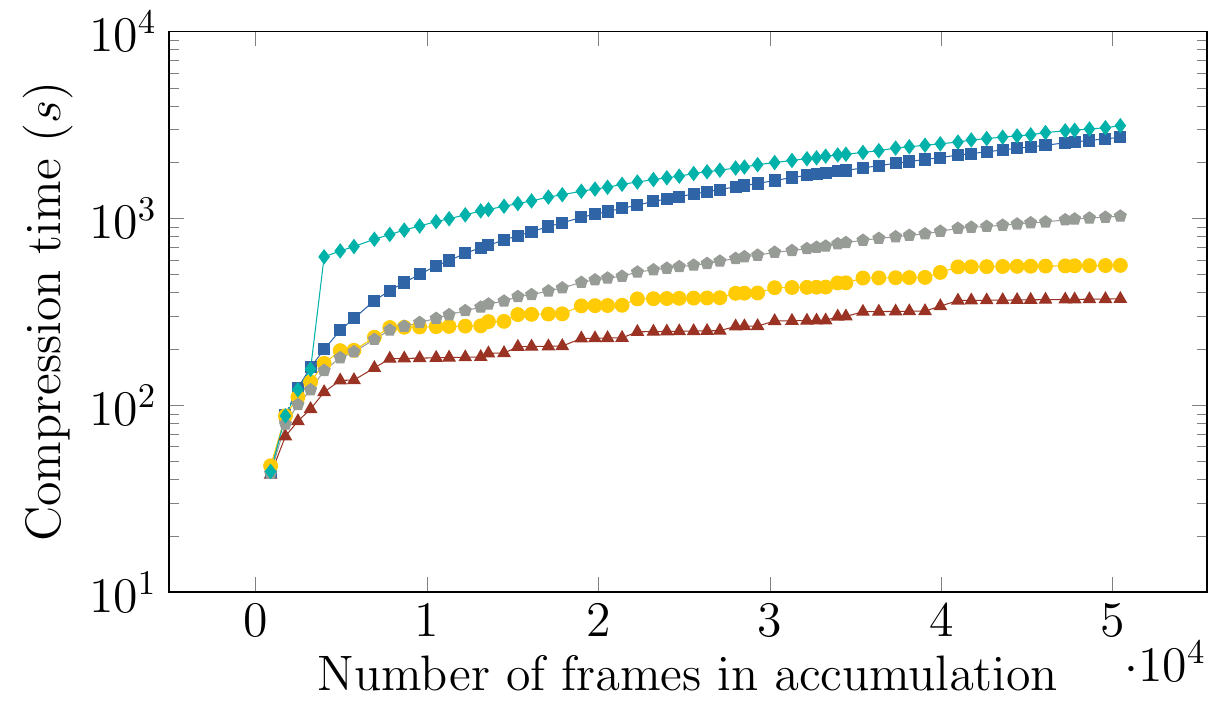}
	\caption{$\varepsilon_{des}=0.1$}
	\label{fig:mspacman01time}
\end{subfigure}
\hfill
\begin{subfigure}{0.49\textwidth}
	\centering
	\includegraphics[width=\textwidth]{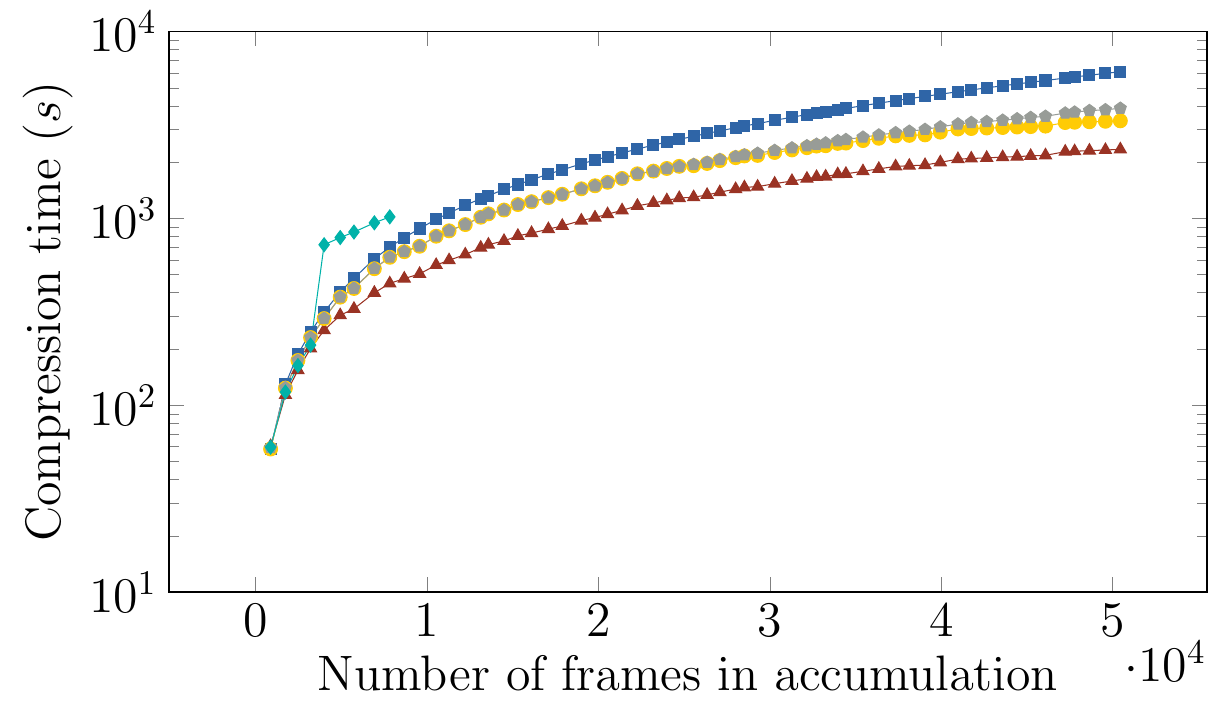}
	\caption{$\varepsilon_{des}=0.01$}
	\label{fig:mspacman001time}
\end{subfigure}
\mcaption{Comparison of execution time vs total number of observations in the tensor train for different versions of TT-ICE, TT-ICE$^*$ and ITTD\cite{Liu2018} with two different $\varepsilon$ settings. TT-ICE$^*$: full TT-ICE$^*$ algorithm with all heuristics, TT-ICE: TT-ICE, Subselect: subselecting frames from runs, SS: subselecting frames from runs and skip updating the first $d$ TT-cores, ITTD 5: ITTD with rounding at every fifth increment step.
ITTD has worst compression time while TT-ICE$^*$ with all heuristics performs best. ITTD Fails to compress the entire stream for $\varepsilon_{des}=0.01$.}
\label{fig:mspacmantime}
\end{figure}

\Cref{fig:mspacman01error,fig:mspacman001error} investigate a difference in the representation error of the TT-cores trained with different algorithms. For each case, we present the mean RRE over the compressed portion of the training set and the mean RPE over the entire validation set after each increment.
Figure~\ref{fig:mspacman01error} shows that all methods can successfully represent the streamed data within the desired relative error upper bound $\varepsilon_{des}$. Furthermore, Figure~\ref{fig:mspacman01error} depicts that the mean RPE of TT-ICE and TT-ICE$^*$ converge asymptotically to the mean RRE. 
Since the mean RPE represents the representation quality of the validation set, these results indicate that a suitable basis is found prior to observing the validation dataset itself. Once the basis to represent observations within the desired accuracy is complete, TT-ICE will not be able to find any unique orthogonal directions to expand the bases of TT-cores for a stream.

\Cref{fig:mspacman01error,fig:mspacman001error} present the mean RPE for ITTD5 only at the steps where rounding is performed and do not connect the data points with dashed line. This is caused by the fact that without reorthogonalization, TT-cores obtained through ITTD are not suitable for assimilating data (projection and prediction). This is an artifact caused by implementing addition in TT-format and the uncontrolled rank inflation and lack of core orthogonality. However, an interesting point from Figure~\ref{fig:mspacman01error} is that the mean RPE falls well below the mean RRE for ITTD5 after reorthogonalization. The results for TT-ICE and TT-ICE$^*$ are similar for $\varepsilon_{des}=0.01$ in Figure~\ref{fig:mspacman001error}. Unfortunately, ITTD5 cannot display the same pattern in this case, since it terminates prematurely due to insufficient memory.

\begin{figure}[h!]
	\begin{subfigure}{\textwidth}
		\centering
		\includegraphics[width=0.5\textwidth]{pacman-legend}
	\end{subfigure}
\begin{subfigure}{0.49\textwidth}
	\centering
	\includegraphics[width=\textwidth]{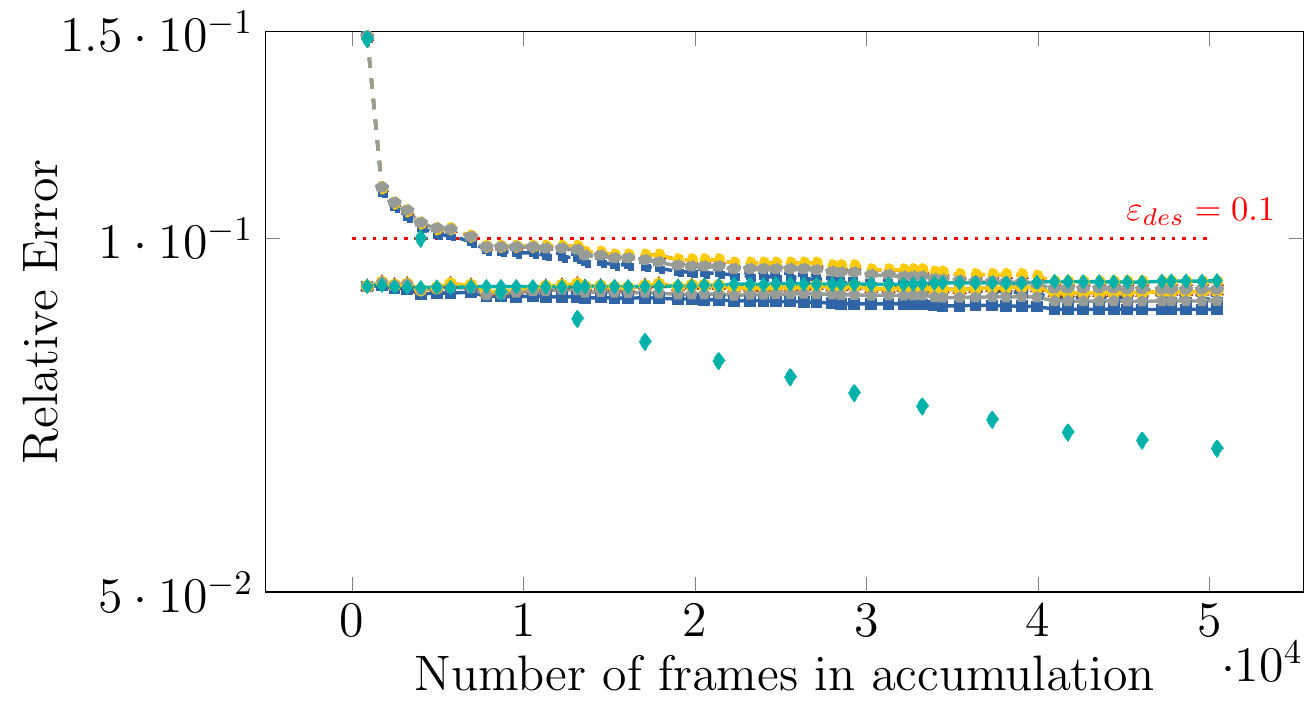}
	\caption{$\varepsilon_{des}=0.1$ (Dotted line)}
	\label{fig:mspacman01error}
\end{subfigure}
\hfill
\begin{subfigure}{0.49\textwidth}
	\centering
	\includegraphics[width=\textwidth]{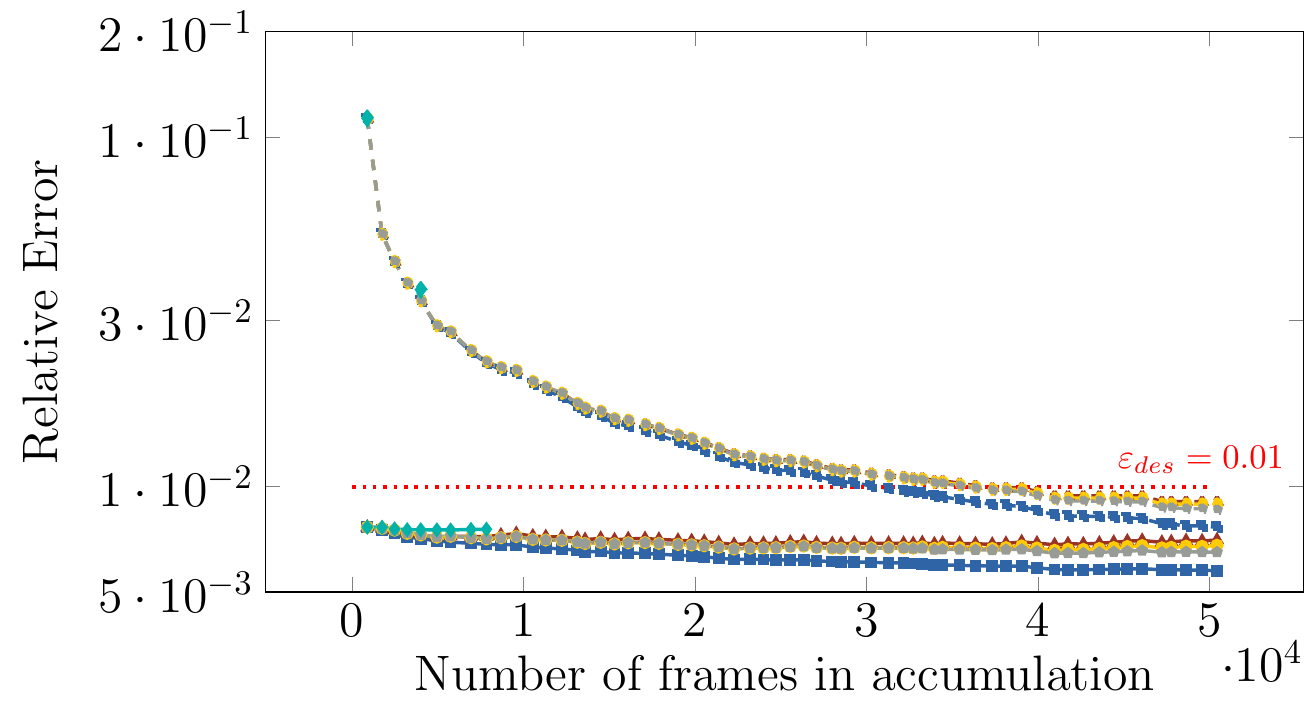}
	\caption{$\varepsilon_{des}=0.01$ (Dotted line)}
	\label{fig:mspacman001error}
\end{subfigure}
\mcaption{Comparison of mean RRE and mean RPE vs total number of observations in the tensor train for different versions of TT-ICE, TT-ICE$^*$ and ITTD\cite{Liu2018} with two different $\varepsilon$ settings. We present the mean RRE over compressed runs in the accumulation and mean RPE over the validation set after each increment step. Data points connected with a solid line represent the mean RRE and data points connected with a dashed line represent the mean RPE. Please refer to Figure~\ref{fig:mspacmantime} for a description of the legend. For $\varepsilon_{des}=0.1$ TT-cores returned from ITTD have much lower RPE than desired. RPE of TT-cores trained with TT-ICE and TT-ICE$^*$ converge to RRE as accumulation size increases. ITTD Fails to compress the entire stream for $\varepsilon_{des}=0.01$.}
\label{fig:mspacmanerror}
\end{figure}

Finally, \Cref{fig:mspacman01compression,fig:mspacman001compression} indicate differences in compression ratio between algorithms. Figure~\ref{fig:mspacman01compression} shows that the TT-ICE variations have comparable performance to each other. Furthermore, Figure~\ref{fig:mspacman01compression} shows selecting any TT-ICE variation results in almost two orders of magnitude higher compression over ITTD5. This significantly lower compression ratio of ITTD also explains its reduction in mean RPE after TT-rounding. The superfluous increase in TT-ranks allows the TT-cores trained with ITTD to cover a much larger portion of the multidimensional basis. Right after reorthogonalization, this larger basis provides an increase in the generalization capability of the TT-cores but also results in a much lower compression ratio. The rounding step also provides a positive jump in the compression ratio, but when the TT-cores are incremented further with ITTD, this improvement decays quickly. When the relative error tolerance changes to $\varepsilon_{des}=0.01$, the low compression ratio problem for ITTD5 is exacerbated where the compression ratio falls below the critical value of 1 for ITTD5. Having a compression ratio lower than 1 means that TT-cores have more entries than the original accumulation. Despite having a much lower compression ratio than the case with $\varepsilon_{des}=0.1$, all of the TT-ICE variations again perform comparably with a compression ratio around $3.5\times$.

To summarize, the investigations yielded a decrease in compression time, an increase in compression ratio, and improved execution stability when TT-ICE is preferred over ITTD. Furthermore, investigations show an additional decrease in compression time and an increase in compression ratio when TT-ICE$^*$ is preferred over TT-ICE.

\begin{figure}[h!]
	\begin{subfigure}{\textwidth}
		\centering
		\includegraphics[width=0.49\textwidth]{pacman-legend}
	\end{subfigure}
	\begin{subfigure}{0.5\textwidth}
	\centering
	\includegraphics[width=\textwidth]{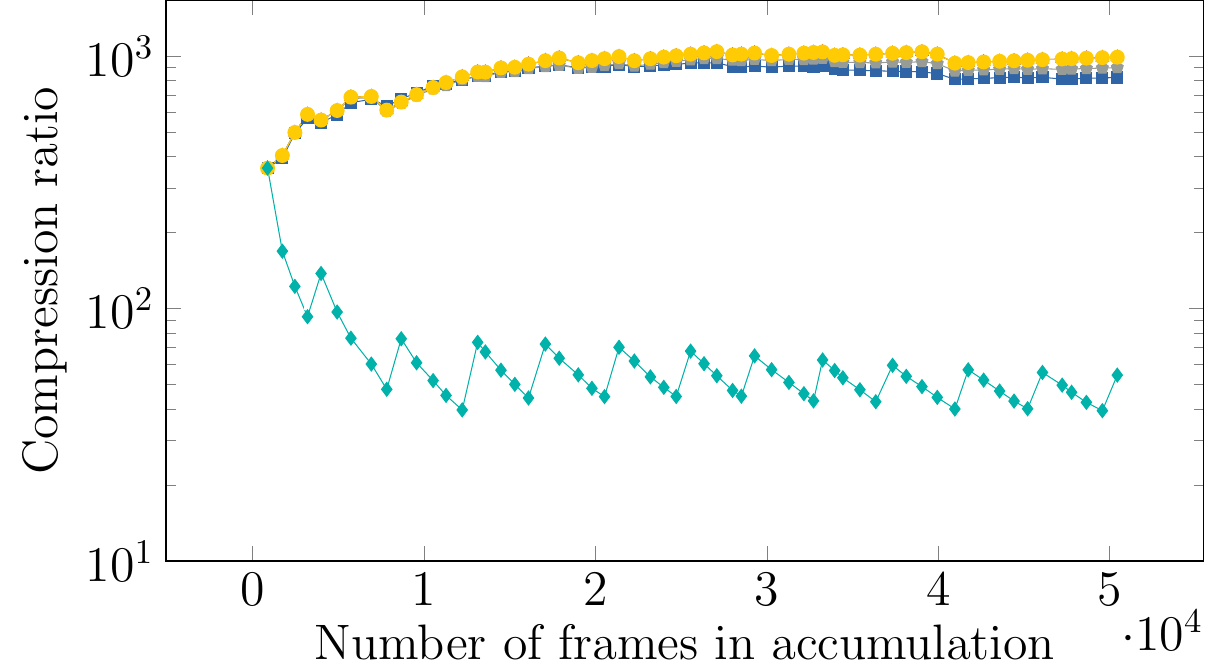}
	\caption{$\varepsilon_{des}=0.1$}
	\label{fig:mspacman01compression}
	\end{subfigure}
\hfill
	\begin{subfigure}{0.49\textwidth}
	\centering
	\centering
	\includegraphics[width=\textwidth]{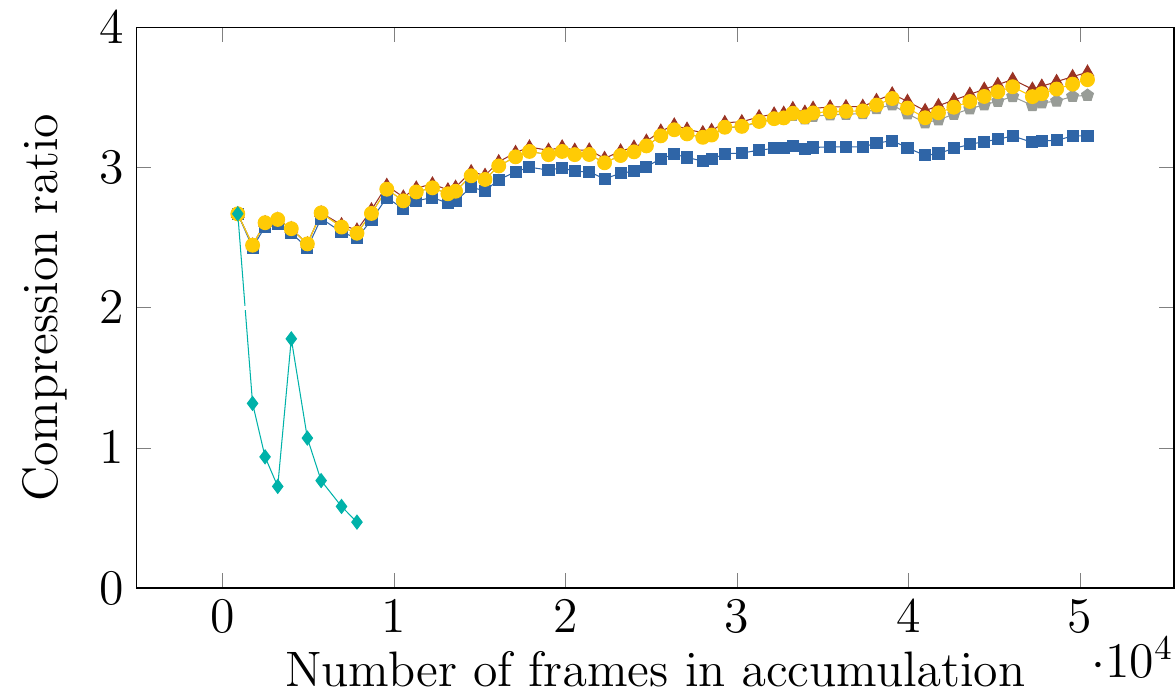}
	\caption{$\varepsilon_{des}=0.01$}
	\label{fig:mspacman001compression}
	\end{subfigure}
\mcaption{Comparison of compression ratio vs total number of observations in the tensor train for different versions of TT-ICE, TT-ICE$^*$, and ITTD\cite{Liu2018} with two different $\varepsilon$ settings. Please refer to Figure~\ref{fig:mspacmantime} for a description of the legend. ITTD has the worst compression ratio while TT-ICE$^*$ with all heuristics performs the best. ITTD Fails to compress the entire stream for $\varepsilon_{des}=0.01$.}
\label{fig:mspacmancompression}
\end{figure}

\subsubsection{Self-oscillating gel simulations}
This section includes tests of TT-ICE's performance to compress simulation outputs of high-dimensional PDEs.

When high-dimensional systems are of interest, the size of the simulation outputs can become prohibitive. In the extreme, even the storage of the data may not be feasible. Streaming compression algorithms can become essential when the outputs of a dynamical simulation become sequentially available. Moreover, tensor decompositions can extract more information than incremental matrix decomposition methods to be used to learn low-dimensional representations, for example, for inverse design~\cite{Aksoy2022}.

The second dataset arises from solutions of a parametric PDE that simulates the motion of a hexagonal sheet of self-oscillating gels.
The motion of the gel is governed by the following time-dependent parametric PDE
\begin{align}\label{eq:catgelparametricpde}
	\mu\frac{\partial r}{\partial t}=f_{s}\left(r,\mathbf{K_{s}},\eta\right)+f_{B}\left(r\right), && \eta\left(x,y,t,\mathbf{A},\mathbf{k}\right)=1+\mathbf{A}\sin\left(2\pi\left(\mathbf{k}\sqrt{x^{2}+y^{2}}-t\right)\right),
\end{align}
where the bold terms indicate the input parameters to the forward model that define the characteristics of the excitation as well as the mechanical properties of the gel. More specifically, $\mathbf{K_s}$ denotes the stretching stiffness of the sheet, $\mathbf{k}$ determines the wavenumber of the sinusoidal excitation, and $\mathbf{A}$ determines the amplitude of the wave traveling on the sheet. Other terms governing the overdamped sheet dynamics are: internal damping coefficient $\mu$, material coordinates $r=(x,y,z)$, stretching force $f_s$, bending force $f_b$, rest strain $\eta$, and time $t$.

We uniformly sample from this 3-dimensional space of input parameters to obtain 6400 unique parameter combinations and then simulate each of those parameter combinations using the approach in~\cite{Alben2019}. The simulations are chaotic, but we use 10 sequential timesteps from each simulation as our data. Specifically, we seek to compress the $x,y$, and $z$ coordinates of 3367 mesh nodes on a hexagonal gel sheet for each time snapshot as shown in Figure~\ref{fig:examplegels}.

To summarize, the data consists of $3367\times3\times10$ tensors for \textit{each} parameter combination that contain the coordinate information of the mesh. We refer to those output tensors as \textit{simulations} for brevity and treat them as individual incremental units.
For compression experiments, each simulation is reshaped into tensors of size $7\times13\times37\times3\times10\times1$ and stacked along the sixth dimension to obtain the accumulation. We conducted all the experiments on M3 machine that has a Xeon Silver 4110 processor and 16GB memory. \Cref{fig:catgeltime,fig:catgelcompression,fig:catgelerror} show the results of those experiments. The occupancy threshold of TT-ICE$^*$ is set to 1 for this dataset. This prevents update attempts when the basis for one dimension is fully explored (i.e. the TT-core has full rank). Finally, we compute the mean of RRE over compressed simulations in the accumulation. 

Despite the suitability of TT-FOA to this tensor stream, we can't provide any results for TT-FOA. Even when we assume that TT-FOA is initialized with the final ranks of TT-ICE${}^{*}$ with $\varepsilon_{des}=0.1$, the memory required by the auxiliary matrices, matrix inversions, and Kronecker products in the algorithm exceeds the memory of M3 machine and fails at the first step\footnote{When we repeat the same experiment with M1 machine, which has the same processor with M3 and has higher memory, computing one step of the stream takes more than $1500s$.}.

\paragraph{\textbf{Compression results}}

This section includes the results of experiments on self-oscillating gel simulation snapshots.

Unlike the previous set of experiments involving ATARI gameplay frames, ITTD methods fail to compress the entire stream in both $\varepsilon_{des}$ settings due to insufficient memory while attempting TT-rounding. In all scenarios, TT-ICE$^*$ outperforms TT-ICE, but both algorithms are able to compress the entire stream successfully.

\Cref{fig:catgel01time,fig:catgel001time} depict the differences in compression speed between algorithms. Figure~\ref{fig:catgel01time} shows that it takes more time for both ITTD2 and ITTD5 to compress $1/20$ of the entire stream than for TT-ICE and TT-ICE$^*$ to compress the entire stream. Figure~\ref{fig:catgel01time} also clearly illustrates the benefits of implementing heuristic upgrades, since TT-ICE$^*$ compresses the entire stream an order of magnitude faster than TT-ICE. This difference in compression time becomes less in Figure~\ref{fig:catgel001time} with $\varepsilon_{des}=0.01$ but TT-ICE$^*$ still provides a significant reduction in time in comparison to TT-ICE. At $\varepsilon_{des}=0.01$, ITTD methods fail even earlier than before and compress at most $1/50$ of the entire stream.

\begin{figure}[h!]
	\begin{subfigure}{\textwidth}
		\centering
		\includegraphics[width=0.5\textwidth]{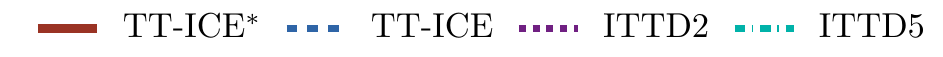}
	\end{subfigure}
\begin{subfigure}{0.49\textwidth}
	\centering
	\includegraphics[width=\textwidth]{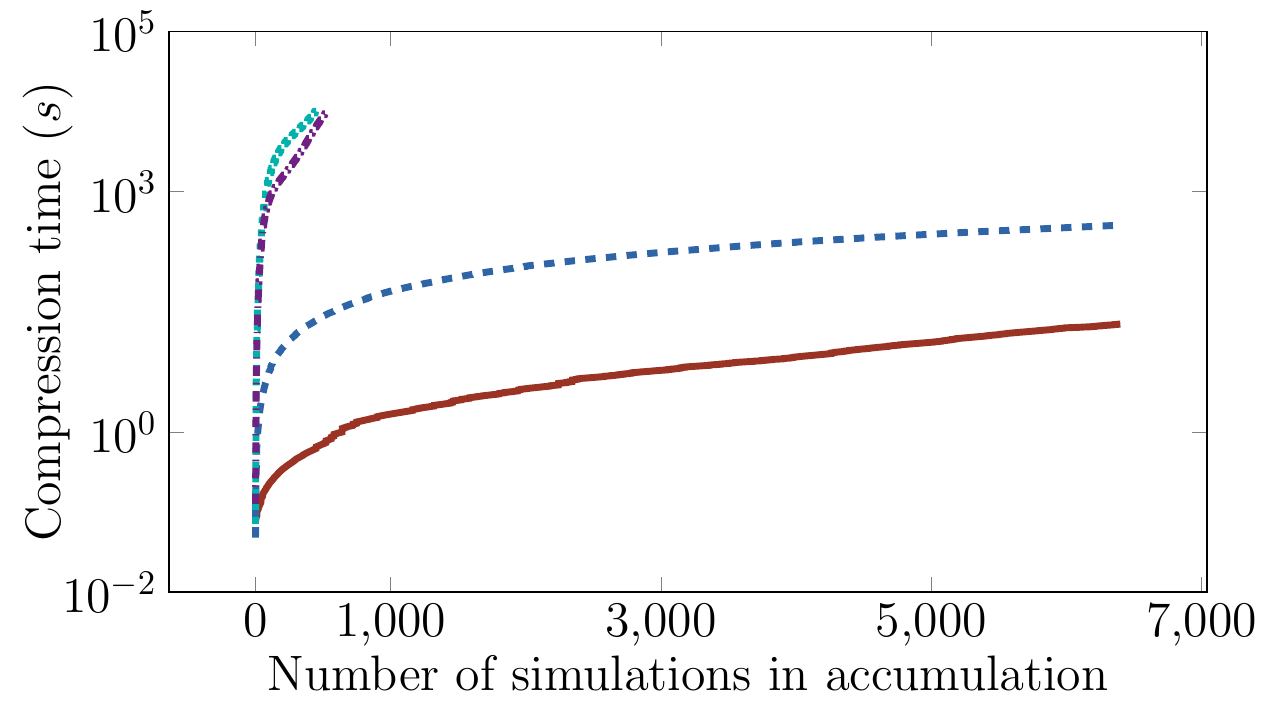}
	\caption{$\varepsilon_{des}=0.1$}
	\label{fig:catgel01time}
\end{subfigure}
\hfill
\begin{subfigure}{0.49\textwidth}
	\centering
	\includegraphics[width=\textwidth]{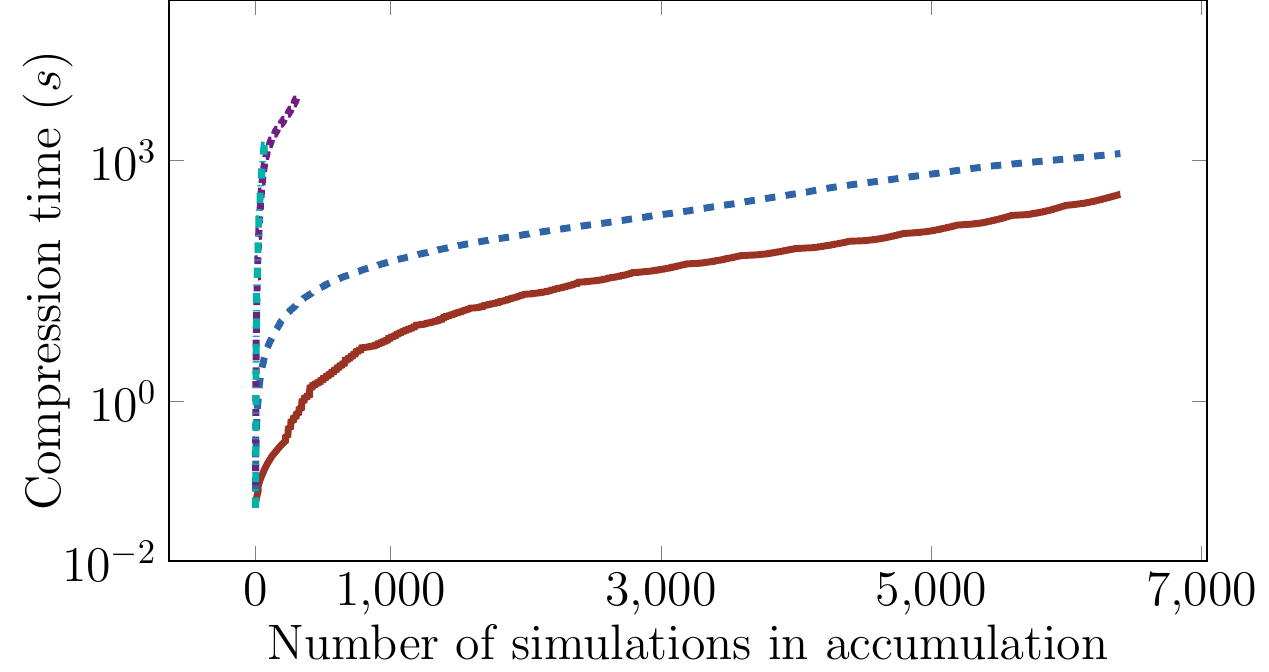}
	\caption{$\varepsilon_{des}=0.01$}
	\label{fig:catgel001time}
\end{subfigure}
\mcaption{Comparison of execution time vs total number of simulations in the tensor train for Algorithm~\ref{alg:TT-ICE}, Algorithm~\ref{alg:coreupdateheuristic} and ITTD\cite{Liu2018} with two different $\varepsilon$ settings. TT-ICE$^*$ outperforms TT-ICE by one order of magnitude. ITTD Fails to compress the entire stream for both cases. TT-ICE$^*$: full TT-ICE$^*$ with all heuristics defined in Section~\ref{sec:heuristics}, TT-ICE: TT-ICE algorithm (Algorithm~\ref{alg:TT-ICE}), ITTD5: ITTD with rounding at every fifth increment step, ITTD2: ITTD with rounding at every second increment step.}
\label{fig:catgeltime}
\end{figure}

\Cref{fig:catgel01compression,fig:catgel001compression} present the differences in the compression ratio between algorithms. Figure~\ref{fig:catgel01error} shows that both TT-ICE and TT-ICE$^*$ achieve excessive compression close to $10^{4}\times$ in the early stages of the stream and then exhibit a decay in compression ratio. This behavior is related to the streaming order of the simulations and is expected. Towards the beginning, the stream consists of simulations with similar parameter combinations. This leads to simulations exhibiting similar motion patterns and allows the accumulation to be represented with a small basis. Then, as the stream progresses, the accumulation consists of simulations from a greater variety of parameter combinations and calls for an expansion in the bases. The same pattern repeats itself in Figure~\ref{fig:catgel001compression} for $\varepsilon_{des}=0.01$, but this time the decay is greater than it was for $\varepsilon_{des}=0.1$. Since both ITTD methods fail to compress the entire stream, we can not make meaningful comments on their compression performance. However, \Cref{fig:catgel01compression,fig:catgel001compression} indicate that the peak compression achieved is consistently two orders of magnitude less than that achieved by the TT-ICE methods.

\begin{figure}[h!]
	\begin{subfigure}{\textwidth}
		\centering
		\includegraphics[width=0.5\textwidth]{catgel-legend}
	\end{subfigure}
\begin{subfigure}{0.49\textwidth}
	\centering
	\includegraphics[width=\textwidth]{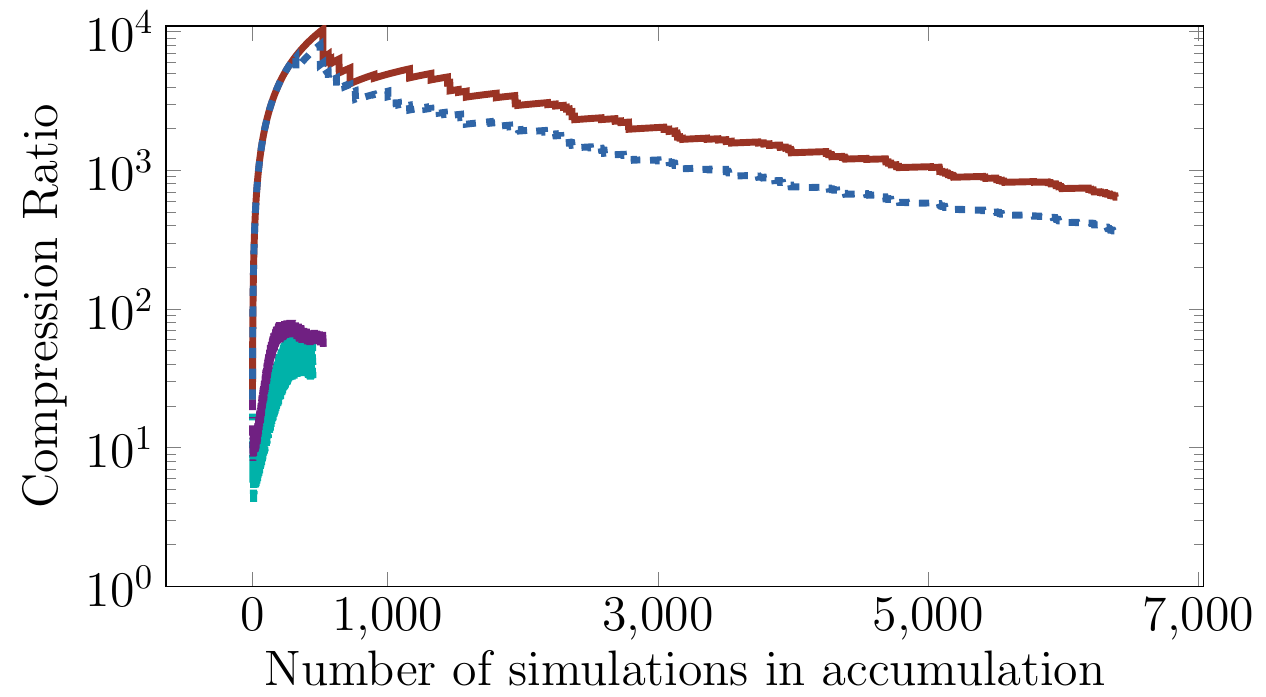}
	\caption{$\varepsilon_{des}=0.1$}
	\label{fig:catgel01compression}
\end{subfigure}
\hfill
\begin{subfigure}{0.49\textwidth}
	\centering
	\includegraphics[width=\textwidth]{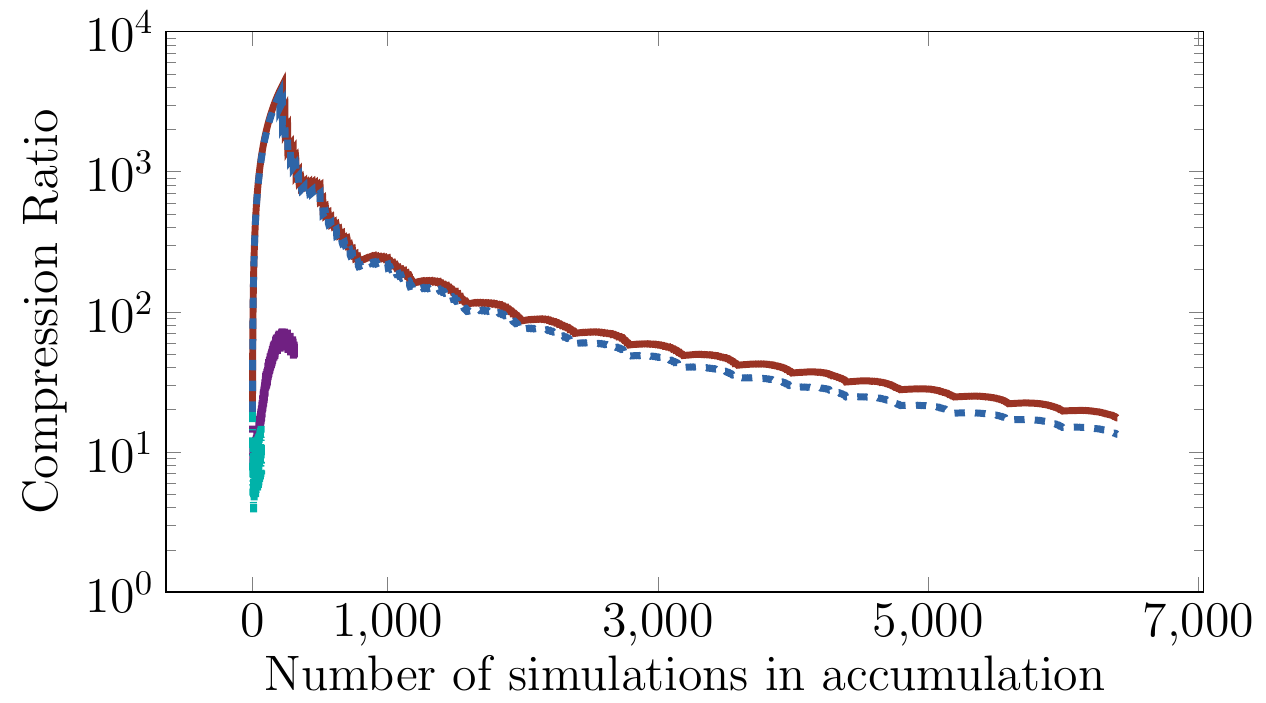}
	\caption{$\varepsilon_{des}=0.01$}
	\label{fig:catgel001compression}
\end{subfigure}
\mcaption{Comparison of compression ratio vs total number of simulations in the tensor train for Algorithm~\ref{alg:TT-ICE}, Algorithm~\ref{alg:coreupdateheuristic} and ITTD\cite{Liu2018} with two different $\varepsilon$ settings. Please refer to Figure~\ref{fig:catgeltime} for a description of the legend. TT-ICE$^*$ outperforms TT-ICE$^*$ but both methods show comparable compression performances. ITTD Fails to compress the entire stream for both cases. }
\label{fig:catgelcompression}
\end{figure}

Finally, \Cref{fig:catgel01error,fig:catgel001error} indicate differences in reconstruction error between algorithms.  Figure~\ref{fig:catgel01error} shows that TT-ICE$^*$ has a mean RRE slightly closer to $\varepsilon_{des}$ than TT-ICE. On the other hand, Figure~\ref{fig:catgel001error} shows that TT-ICE$^*$ has nearly twice the error of TT-ICE, but this difference in mean RRE diminishes to almost a constant offset as the stream progresses. In both Figure~\ref{fig:catgel01error} and Figure~\ref{fig:catgel001error}, ITTD methods start with lower mean RRE values. This can be explained by the significantly lower compression performance of these methods, where the higher coverage in the bases results in both lower compression and lower mean RRE. However, we cannot draw meaningful conclusions for ITTD since both ITTD2 and ITTD5 fail prematurely.

\begin{figure}[h!]
	\begin{subfigure}{\textwidth}
		\centering
		\includegraphics[width=0.5\textwidth]{catgel-legend}
	\end{subfigure}
\begin{subfigure}{0.48\textwidth}
	\centering
	\includegraphics[width=\textwidth]{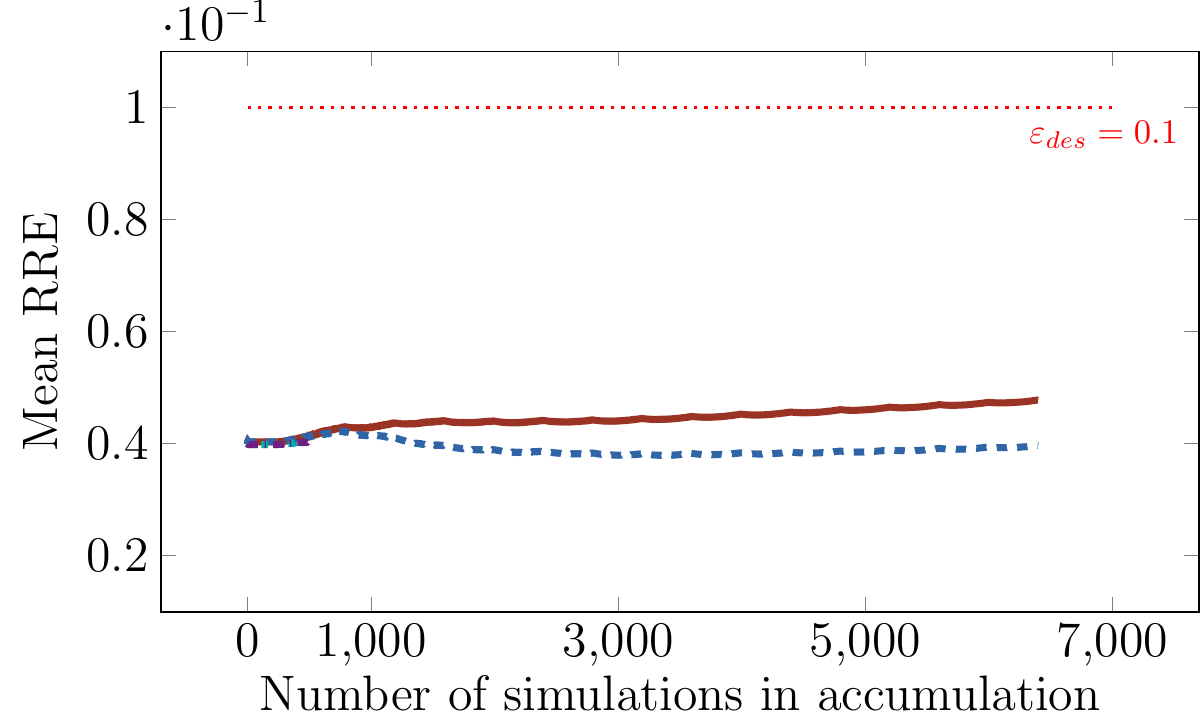}
	\caption{$\varepsilon=0.1$}
	\label{fig:catgel01error}
\end{subfigure}
\hfill
\begin{subfigure}{0.48\textwidth}
	\centering
	\includegraphics[width=\textwidth]{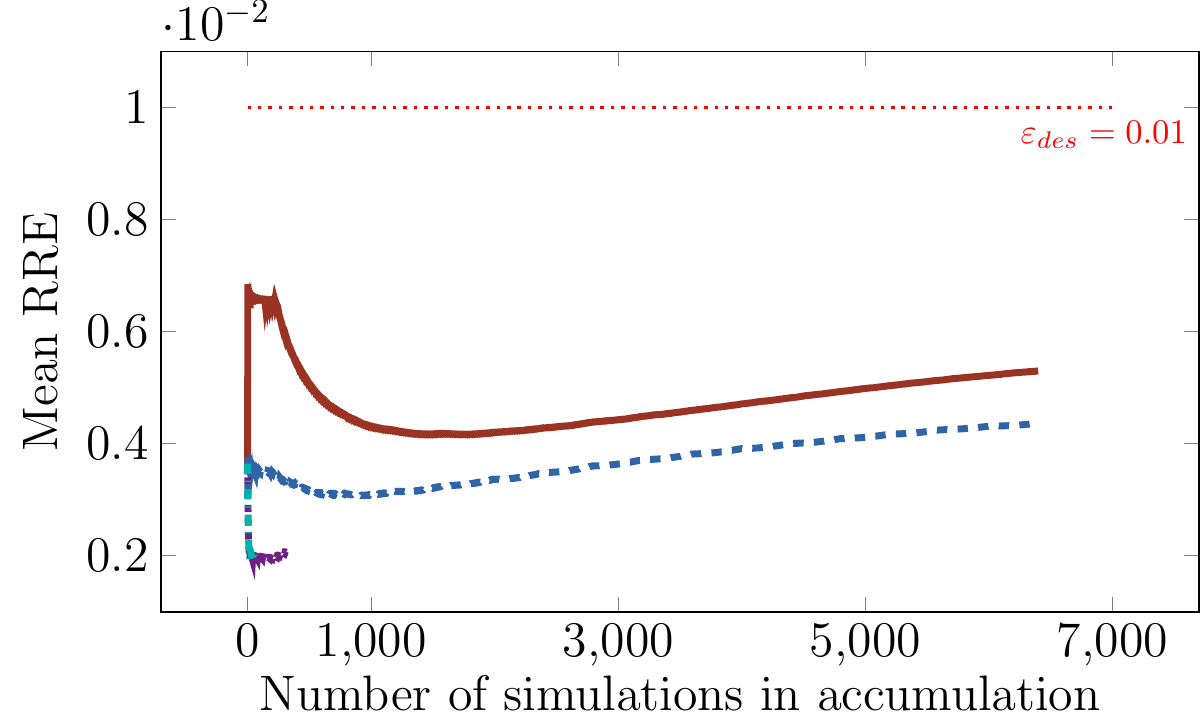}
	\caption{$\varepsilon_{des}=0.01$}
	\label{fig:catgel001error}
\end{subfigure}
\mcaption{Comparison of mean RRE vs total number of simulations in the tensor train for TT-ICE, TT-ICE$^*$ and ITTD\cite{Liu2018} with two different $\varepsilon$ settings. Please refer to Figure~\ref{fig:catgeltime} for a description of the legend. TT-ICE$^*$ provides mean RRE closer to $\varepsilon_{des}$ in both cases. ITTD Fails to compress the entire stream for both cases.}
\label{fig:catgelerror}
\end{figure}

Mean RRE tests conclude the experiments with self-oscillating gel simulations. Similar to experiments with ATARI data, TT-ICE and TT-ICE$^*$ algorithms yield reduced compression time, increased compression ratio, and improved execution stability over ITTD. 

\section{Conclusion}
\label{sec:conclusion}
In this work, we proposed a new algorithm to incrementally update a TT-decomposition to compress a stream of data. Our algorithm TT-ICE improves on the existing state-of-the-art because (1) it maintains a desired error tolerance {\it for all data increments}, (2) it updates its ranks without excessive growth, and (3) it maintains orthogonality of the TT-cores to enable efficient projection and prediction for uncompressed data. We provide proof that the TT-ICE algorithm maintains its accuracy throughout the compression process. We then provide three heuristics to improve on this algorithm and show empirical evidence that they improve performance with little sacrifice in accuracy. This enhanced version of TT-ICE is also guaranteed to maintain the accuracy of the already compressed portion of the stream. However, no such guarantee can be provided for the portion of the stream compressed with TT-ICE${}^{*}$ since heuristics are used in the first place.
Experimental results on two different types of data demonstrate the superior performance of TT-ICE over ITTD and provide empirical proof of the additional benefits obtained from the heuristic upgrades we implemented in TT-ICE$^*$. 
For simulation data, TT-ICE$^*$ achieves twice the compression ratio of TT-ICE in half the time.
For image data, TT-ICE$^*$ achieves a comparable, if not higher, compression ratio with up to 80\% reduction in the TT-ICE time. Moreover, at resource-limited hardware, TT-ICE and TT-ICE$^*$ have proven themselves to be reliable and performant compression methods.

Extensions of this work will attempt to theoretically justify the heuristics used in TT-ICE$^*$. It will also deploy the proposed methods to enable scalable machine learning in the latent space identified by the compression, e.g., for inverse design~\cite{Aksoy2022} and behavioral cloning~\cite{chen2021behavioral}.

\section*{Acknowledgments}
DA, AG, and SV acknowledge partial support from the Automotive Research Center at the University of Michigan (UM) in accordance with Cooperative Agreement W56HZV-19-2-0001 with U.S. Army DEVCOM Ground Vehicle Systems Center. 
DA and AG also acknowledge partial support from the Department of Energy Office of Scientific Research, ASCR, under grant DE-SC0020364. We thank Brian Chen for preparing the ATARI game dataset. The code for TT-ICE and TT-ICE$^*$ is publicly available on github.com/dorukaks/TT-ICE.

A note by DA: {\em This paper is dedicated to the loving memory of my grandmother Ayla Ya\c{s}ar (1942-2022). I will miss making you Turkish coffee.}

\bibliographystyle{siamplain}  
\bibliography{REReferences,AuxReferences}

\begin{thebibliography}{10}

\bibitem{Aksoy2022}
{\sc D.~Aksoy, S.~Alben, R.~D. Deegan, and A.~A. Gorodetsky}, {\em {Inverse design of self-oscillatory gels through deep learning}}, Neural Computing and Applications, 34 (2022), pp.~6879--6905, \url{https://doi.org/10.1007/s00521-021-06788-9}, \url{https://doi.org/10.1007/s00521-021-06788-9}.

\bibitem{Alben2019}
{\sc S.~Alben, A.~A. Gorodetsky, D.~Kim, and R.~D. Deegan}, {\em {Semi-implicit methods for the dynamics of elastic sheets}}, Journal of Computational Physics, 399 (2019), p.~108952, \url{https://doi.org/10.1016/j.jcp.2019.108952}, \url{https://doi.org/10.1016/j.jcp.2019.108952}, \url{https://arxiv.org/abs/1904.09198}.

\bibitem{Anaissi2020}
{\sc A.~Anaissi, B.~Suleiman, and S.~M. Zandavi}, {\em {NeCPD: An Online Tensor Decomposition with Optimal Stochastic Gradient Descent}},  (2020), \url{http://arxiv.org/abs/2003.08844}, \url{https://arxiv.org/abs/2003.08844}.

\bibitem{Baker2012}
{\sc C.~G. Baker, K.~A. Gallivan, and P.~{Van Dooren}}, {\em {Low-rank incremental methods for computing dominant singular subspaces}}, Linear Algebra and Its Applications, 436 (2012), pp.~2866--2888, \url{https://doi.org/10.1016/j.laa.2011.07.018}, \url{http://dx.doi.org/10.1016/j.laa.2011.07.018}.

\bibitem{chen2021behavioral}
{\sc B.~Chen, S.~Tandon, D.~Gorsich, A.~Gorodetsky, and S.~Veerapaneni}, {\em Behavioral cloning in atari games using a combined variational autoencoder and predictor model}, in 2021 IEEE Congress on Evolutionary Computation (CEC), IEEE, 2021, pp.~2077--2084.

\bibitem{cook2019anomaly}
{\sc A.~A. Cook, G.~M{\i}s{\i}rl{\i}, and Z.~Fan}, {\em Anomaly detection for iot time-series data: A survey}, IEEE Internet of Things Journal, 7 (2019), pp.~6481--6494.

\bibitem{de2022tensor}
{\sc S.~De, E.~Corona, P.~Jayakumar, and S.~Veerapaneni}, {\em Tensor-train compression of discrete element method simulation data}, arXiv preprint arXiv:2210.08399,  (2022).

\bibitem{Du2018}
{\sc Y.~Du, Y.~Zheng, K.~C. Lee, and S.~Zhe}, {\em {Probabilistic Streaming Tensor Decomposition}}, Proceedings - IEEE International Conference on Data Mining, ICDM, 2018-November (2018), pp.~99--108, \url{https://doi.org/10.1109/ICDM.2018.00025}.

\bibitem{Eckart1936}
{\sc C.~Eckart and G.~Young}, {\em {The approximation of one matrix by another of lower rank}}, Psychometrika, 1 (1936), pp.~211--218, \url{https://doi.org/10.1007/BF02288367}, \url{http://link.springer.com/10.1007/BF02288367}.

\bibitem{Yu2015}
{\sc M.~H. Engeli}, {\em {Bits and Spaces}}, Architecture and and Computing for Physical, Virtual and Hybrid Realms, 37 (2001), p.~207, \url{https://doi.org/10.5555/ICML.3045145}, \url{http://bitsandspaces.ethz.ch/acknowledgments}.

\bibitem{Giudici2019}
{\sc P.~Giudici, B.~Huang, and A.~Spelta}, {\em {Trade networks and economic fluctuations in Asian countries}}, Economic Systems, 43 (2019), p.~100695, \url{https://doi.org/10.1016/j.ecosys.2019.100695}, \url{https://doi.org/10.1016/j.ecosys.2019.100695}.

\bibitem{Hastad1989}
{\sc J.~H{\aa}stad}, {\em {Tensor rank is NP-complete}}, Lecture Notes in Computer Science (including subseries Lecture Notes in Artificial Intelligence and Lecture Notes in Bioinformatics), 372 LNCS (1989), pp.~451--460, \url{https://doi.org/10.1007/BFb0035776}.

\bibitem{stable-baselines}
{\sc A.~Hill, A.~Raffin, M.~Ernestus, A.~Gleave, A.~Kanervisto, R.~Traore, P.~Dhariwal, C.~Hesse, O.~Klimov, A.~Nichol, M.~Plappert, A.~Radford, J.~Schulman, S.~Sidor, and Y.~Wu}, {\em Stable baselines}.
\newblock \url{https://github.com/hill-a/stable-baselines}, 2018.

\bibitem{Jia2014}
{\sc C.~Jia, Y.~Kong, Z.~Ding, and Y.~Fu}, {\em {Latent tensor transfer learning for RGB-D action recognition}}, in MM 2014 - Proceedings of the 2014 ACM Conference on Multimedia, New York, NY, USA, nov 2014, ACM, pp.~87--96, \url{https://doi.org/10.1145/2647868.2654928}, \url{http://link.springer.com/10.1007/978-1-4419-7142-5_3 https://dl.acm.org/doi/10.1145/2647868.2654928}.

\bibitem{Liu2018}
{\sc H.~Liu, L.~T. Yang, Y.~Guo, X.~Xie, and J.~Ma}, {\em {An Incremental Tensor-Train Decomposition for Cyber-Physical-Social Big Data}}, IEEE Transactions on Big Data, 7 (2018), pp.~341--354, \url{https://doi.org/10.1109/tbdata.2018.2867485}.

\bibitem{Mahoney2011}
{\sc M.~W. Mahoney}, {\em {Randomized algorithms for matrices and data}}, Foundations and Trends in Machine Learning, 3 (2010), pp.~123--224, \url{https://doi.org/10.1561/2200000035}, \url{http://arxiv.org/abs/1104.5557 http://www.nowpublishers.com/article/Details/MAL-035}, \url{https://arxiv.org/abs/1104.5557}.

\bibitem{Martinez-Montes2004}
{\sc E.~Mart{\'{i}}nez-Montes, P.~A. Vald{\'{e}}s-Sosa, F.~Miwakeichi, R.~I. Goldman, and M.~S. Cohen}, {\em {Concurrent EEG/fMRI analysis by multiway Partial Least Squares}}, NeuroImage, 22 (2004), pp.~1023--1034, \url{https://doi.org/10.1016/j.neuroimage.2004.03.038}.

\bibitem{MIRSKY1960}
{\sc L.~Mirsky}, {\em {Symmetric gauge functions and unitarily invariant norms}}, Quarterly Journal of Mathematics, 11 (1960), pp.~50--59, \url{https://doi.org/10.1093/qmath/11.1.50}, \url{https://academic.oup.com/qjmath/article-lookup/doi/10.1093/qmath/11.1.50}.

\bibitem{Miwakeichi2004}
{\sc F.~Miwakeichi, E.~Mart{\'{i}}nez-Montes, P.~A. Vald{\'{e}}s-Sosa, N.~Nishiyama, H.~Mizuhara, and Y.~Yamaguchi}, {\em {Decomposing EEG data into space-time-frequency components using Parallel Factor Analysis}}, NeuroImage, 22 (2004), pp.~1035--1045, \url{https://doi.org/10.1016/j.neuroimage.2004.03.039}.

\bibitem{Mnih2015}
{\sc V.~Mnih, K.~Kavukcuoglu, D.~Silver, A.~A. Rusu, J.~Veness, M.~G. Bellemare, A.~Graves, M.~Riedmiller, A.~K. Fidjeland, G.~Ostrovski, S.~Petersen, C.~Beattie, A.~Sadik, I.~Antonoglou, H.~King, D.~Kumaran, D.~Wierstra, S.~Legg, and D.~Hassabis}, {\em {Human-level control through deep reinforcement learning}}, Nature, 518 (2015), pp.~529--533, \url{https://doi.org/10.1038/nature14236}, \url{http://dx.doi.org/10.1038/nature14236}.

\bibitem{Nakatsuji2017}
{\sc M.~Nakatsuji, Q.~Zhang, X.~Lu, B.~Makni, and J.~A. Hendler}, {\em {Semantic Social Network Analysis by Cross-Domain Tensor Factorization}}, IEEE Transactions on Computational Social Systems, 4 (2017), pp.~207--217, \url{https://doi.org/10.1109/TCSS.2017.2732685}.

\bibitem{Oseledets2011}
{\sc I.~V. Oseledets}, {\em {Tensor-train decomposition}}, SIAM Journal on Scientific Computing, 33 (2011), pp.~2295--2317, \url{https://doi.org/10.1137/090752286}, \url{http://epubs.siam.org/doi/10.1137/090752286}.

\bibitem{rl-zoo}
{\sc A.~Raffin}, {\em Rl baselines zoo}.
\newblock \url{https://github.com/araffin/rl-baselines-zoo}, 2018.

\bibitem{Sizov2010}
{\sc S.~Sizov, S.~Staab, and T.~Franz}, {\em {Analysis of Social Networks by Tensor Decomposition}}, in Handbook of Social Network Technologies and Applications, B.~Furht, ed., Springer US, Boston, MA, 2010, pp.~45--58, \url{https://doi.org/10.1007/978-1-4419-7142-5_3}, \url{http://link.springer.com/10.1007/978-1-4419-7142-5 http://link.springer.com/10.1007/978-1-4419-7142-5_3}.

\bibitem{Smith2018}
{\sc S.~Smith, K.~Huang, N.~D. Sidiropoulos, and G.~Karypis}, {\em {Streaming tensor factorization for infinite data sources}}, in SIAM International Conference on Data Mining, SDM 2018, Philadelphia, PA, may 2018, Society for Industrial and Applied Mathematics, pp.~81--89, \url{https://doi.org/10.1137/1.9781611975321.10}, \url{https://epubs.siam.org/doi/10.1137/1.9781611975321.10}.

\bibitem{Sobral2014}
{\sc A.~Sobral, C.~G. Baker, T.~Bouwmans, and E.~H. Zahzah}, {\em {Incremental and multi-feature tensor subspace learning applied for background modeling and subtraction}}, in Lecture Notes in Computer Science (including subseries Lecture Notes in Artificial Intelligence and Lecture Notes in Bioinformatics), vol.~8814, 2014, pp.~94--103, \url{https://doi.org/10.1007/978-3-319-11758-4_11}, \url{http://link.springer.com/10.1007/978-3-319-11758-4_11}.

\bibitem{Thanh2021}
{\sc L.~T. Thanh, K.~Abed-Meraim, N.~L. Trung, and R.~Boyer}, {\em {Adaptive algorithms for tracking tensor-train decomposition of streaming tensors}}, European Signal Processing Conference, 2021-January (2021), pp.~995--999, \url{https://doi.org/10.23919/Eusipco47968.2020.9287780}.

\bibitem{Vandecappelle2017}
{\sc M.~Vandecappelle, N.~Vervliet, and L.~{De Lathauwer}}, {\em {Nonlinear least squares updating of the canonical polyadic decomposition}}, 25th European Signal Processing Conference, EUSIPCO 2017, 2017-January (2017), pp.~663--667, \url{https://doi.org/10.23919/EUSIPCO.2017.8081290}.

\bibitem{Wang2021}
{\sc X.~Wang, L.~T. Yang, Y.~Wang, L.~Ren, and M.~J. Deen}, {\em {ADTT: A Highly Efficient Distributed Tensor-Train Decomposition Method for IIoT Big Data}}, IEEE Transactions on Industrial Informatics, 17 (2021), pp.~1573--1582, \url{https://doi.org/10.1109/TII.2020.2967768}.

\end{thebibliography}

\appendix
\begin{appendices}
\renewcommand\thetable{\thesection.\arabic{table}}
\renewcommand\thefigure{\thesection.\arabic{figure}}
\section{Other Atari Games}
This section includes the results of the experiments with the other ATARI games and their discussion.

TT-ICE and TT-ICE$^*$ proved their performance in the first experiments with Ms.Pac-Man frames. Next, we investigated the performance of algorithms across different hardware combinations.

For these experiments, we focus our attention on TT-ICE, TT-ICE$^*$, and ITTD and to understand performance across different hardware combinations, we tested the algorithms on several machines. For each experiment, we detail the machine on which it was conducted. The machines are \textit{Dell Precision 7820} workstation with Xeon Silver 4110 processor and 32 GB memory, \textit{Dell XPS-15 9570} laptop with i7-8750H processor and 16 GB memory, and \textit{Dell Precision 7820} workstation with Xeon Silver 4110 processor and 16 GB memory. The machines are referred to in Table~\ref{tab:atariresults} as M1, M2, and M3, respectively.
Table~\ref{tab:atariresults} indicates improvements between $
44.5\times$ to $56.7\times$ in compression and between $8.2\times$ to $23.5\times$ in time over ITTD5.

In Table~\ref{tab:atariresults}, TT-ICE and TT-ICE$^*$ complete the compression of streams faster than ITTD5. Furthermore, our proposed algorithms achieve at least an order of magnitude higher compression ratio than ITTD5. In Table~\ref{tab:atariresults}, ITTD5 fails without an exception to compress the entire stream due to insufficient memory when $\varepsilon_{des}=0.01$. This is caused by the high computational requirements of the rounding step in ITTD5 and becomes prohibitive even with $\varepsilon_{des}=0.1$ for games Beamrider and Enduro. The only case where all three methods fail to compress the entire stream is Enduro with $\varepsilon_{des}=0.01$ due to the limited memory of the machine. However, even in that case, we see that TT-ICE can compress at least 3 times more frames than ITTD5.

An impressive point from Table~\ref{tab:atariresults} is the difference in number of frames used between TT-ICE and TT-ICE$^*$. By subselecting frames and skipping the update of the first $d$ TT-cores, TT-ICE$^*$ reduces the number of frames used in updates by up to $97\%$ (Pong with $\varepsilon_{des}=0.1$) without any loss in mean RRE. The reduction in time is also parallel with the reduction in the frames used. TT-ICE$^*$ achieves at least $48\%$ reduction (Seaquest with $\varepsilon_{des}=0.01$) in execution time over TT-ICE.

In Table~\ref{tab:atariresults}, Beamrider with $\varepsilon=0.01$ is another insightful example, which shows that implementing heuristics can be influential on the success of TT-ICE. In that case, TT-ICE$^*$ successfully completes the decomposition of the stream, whereas TT-ICE fails to complete the same task.

The additional experiments using ATARI data show that TT-ICE and TT-ICE$^*$ can successfully compress streams across different hardware combinations. Furthermore, they provide proof that the decrease in compression time over ITTD is not hardware dependent.

	To provide a qualitative perspective on what each of the $\varepsilon_{des}$ levels corresponds, we provide reconstructions of the frames compressed both with $\varepsilon_{des}=0.1$ (\Cref{fig:eps01reconstruction}) and $\varepsilon_{des}=0.01$ (\Cref{fig:eps001reconstruction}). In \Cref{fig:eps01reconstruction} we include ITTD5 reconstructions in addition to TT-ICE reconstructions. \Cref{fig:eps01reconstruction} shows that there are visual distortions on each reconstructed frame regardless of the algorithm used. These artifacts caused at the error tolerance of $\varepsilon_{des}=0.1$ can be as unobtrusive as discolored parts/objects (Breakout, Ms.Pac-Man) but also can be completely disruptive as invisible agents (Pong, Q*bert).
		
	\begin{figure}[h!]
		\begin{subfigure}{0.12\textwidth}
			\includegraphics[width=\columnwidth]{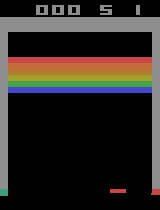}
		\end{subfigure}
		\begin{subfigure}{0.12\textwidth}
			\includegraphics[width=\columnwidth]{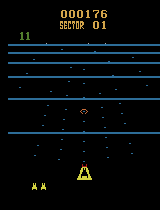}
		\end{subfigure}
		\begin{subfigure}{0.12\textwidth}
			\includegraphics[width=\columnwidth]{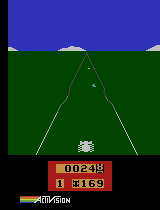}
		\end{subfigure}
		\begin{subfigure}{0.12\textwidth}
			\includegraphics[width=\columnwidth]{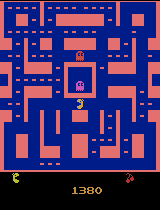}
		\end{subfigure}
		\begin{subfigure}{0.12\textwidth}
			\includegraphics[width=\columnwidth]{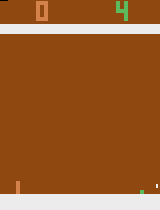}
		\end{subfigure}
		\begin{subfigure}{0.12\textwidth}
			\includegraphics[width=\columnwidth]{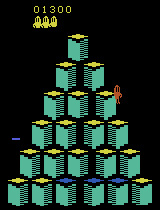}
		\end{subfigure}
		\begin{subfigure}{0.12\textwidth}
			\includegraphics[width=\columnwidth]{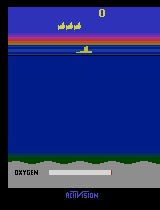}
		\end{subfigure}
		\begin{subfigure}{0.12\textwidth}
			\includegraphics[width=\columnwidth]{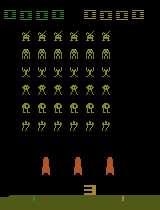}
		\end{subfigure}
	
		\begin{subfigure}{0.12\textwidth}
			\includegraphics[width=\columnwidth]{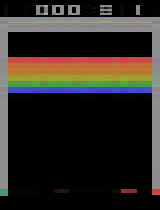}
		\end{subfigure}
		\begin{subfigure}{0.12\textwidth}
			\includegraphics[width=\columnwidth]{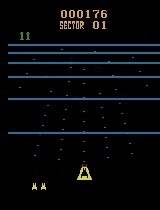}
		\end{subfigure}
		\begin{subfigure}{0.12\textwidth}
			\includegraphics[width=\columnwidth]{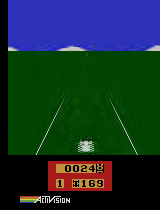}
		\end{subfigure}
		\begin{subfigure}{0.12\textwidth}
			\includegraphics[width=\columnwidth]{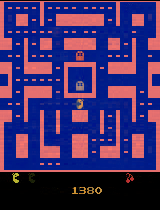}
		\end{subfigure}
		\begin{subfigure}{0.12\textwidth}
			\includegraphics[width=\columnwidth]{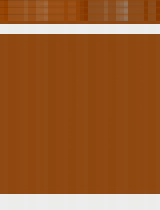}
		\end{subfigure}
		\begin{subfigure}{0.12\textwidth}
			\includegraphics[width=\columnwidth]{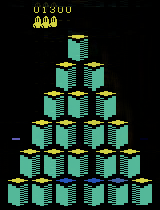}
		\end{subfigure}
		\begin{subfigure}{0.12\textwidth}
			\includegraphics[width=\columnwidth]{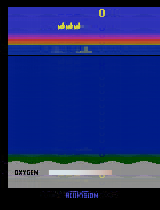}
		\end{subfigure}
		\begin{subfigure}{0.12\textwidth}
			\includegraphics[width=\columnwidth]{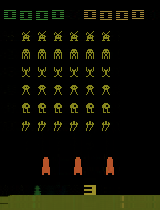}
		\end{subfigure}
	
		\begin{subfigure}{0.12\textwidth}
		\includegraphics[width=\columnwidth]{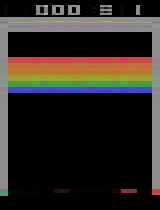}
		\end{subfigure}
		\begin{subfigure}{0.12\textwidth}
		\includegraphics[width=\columnwidth]{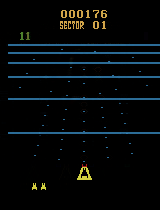}
		\end{subfigure}
		\begin{subfigure}{0.12\textwidth}
		\includegraphics[width=\columnwidth]{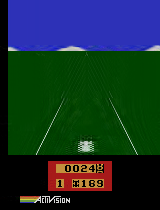}
		\end{subfigure}
		\begin{subfigure}{0.12\textwidth}
		\includegraphics[width=\columnwidth]{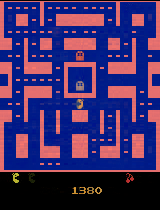}
		\end{subfigure}
		\begin{subfigure}{0.12\textwidth}
		\includegraphics[width=\columnwidth]{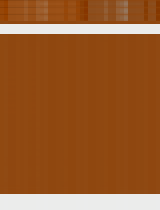}
		\end{subfigure}
		\begin{subfigure}{0.12\textwidth}
		\includegraphics[width=\columnwidth]{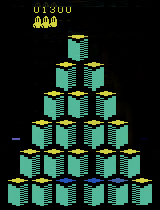}
		\end{subfigure}
		\begin{subfigure}{0.12\textwidth}
		\includegraphics[width=\columnwidth]{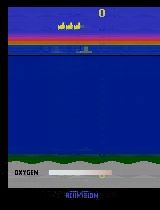}
		\end{subfigure}
		\begin{subfigure}{0.12\textwidth}
		\includegraphics[width=\columnwidth]{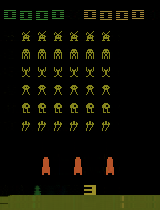}
		\end{subfigure}
		\caption{ Original frames from the Atari dataset (top) along with their respective reconstructions from the TT-cores compressed using TT-ICE (middle) and ITTD5 (bottom) with $\varepsilon_{des}=0.1$. There are visible visual artifacts on some of the frames. There is no visible difference between TT-ICE and ITTD5 reconstructions. Left-to-right: Breakout, Beam Rider, Enduro, Ms.Pac-Man, Pong, Q*bert, Seaquest, Space Invaders
		}
		\label{fig:eps01reconstruction}
	\end{figure}

	As the error tolerance is reduced to $\varepsilon_{des}=0.01$, the visual artifacts in the reconstruction disappear. In \Cref{fig:eps001reconstruction}, we see that for $\varepsilon_{des}=0.01$ the frames can be reconstructed with no visible disruptions. This also provides evidence that TT-cores trained with TT-ICE and TT-ICE${}^{*}$ can be used as a method of storage. Unfortunately at this error tolerance level, we can not provide reconstructions of ITTD, as ITTD5 fails due to insufficient memory for all games.
	
	\begin{figure}[h!]
		\begin{subfigure}{0.12\textwidth}
			\includegraphics[width=\columnwidth]{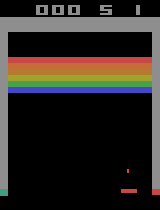}
		\end{subfigure}
		\begin{subfigure}{0.12\textwidth}
			\includegraphics[width=\columnwidth]{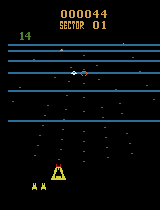}
		\end{subfigure}
		\begin{subfigure}{0.12\textwidth}
			\includegraphics[width=\columnwidth]{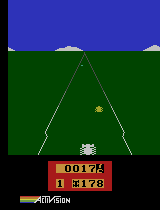}
		\end{subfigure}
		\begin{subfigure}{0.12\textwidth}
			\includegraphics[width=\columnwidth]{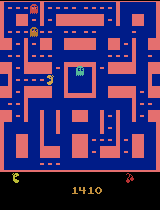}
		\end{subfigure}
		\begin{subfigure}{0.12\textwidth}
			\includegraphics[width=\columnwidth]{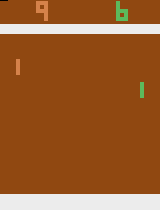}
		\end{subfigure}
		\begin{subfigure}{0.12\textwidth}
			\includegraphics[width=\columnwidth]{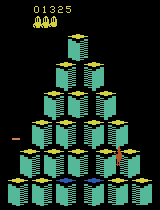}
		\end{subfigure}
		\begin{subfigure}{0.12\textwidth}
			\includegraphics[width=\columnwidth]{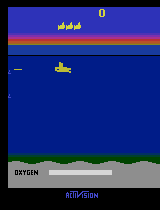}
		\end{subfigure}
		\begin{subfigure}{0.12\textwidth}
			\includegraphics[width=\columnwidth]{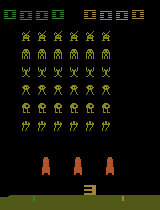}
		\end{subfigure}
		\begin{subfigure}{0.12\textwidth}
			\includegraphics[width=\columnwidth]{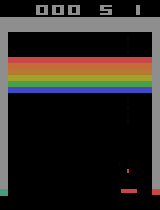}
		\end{subfigure}
		\begin{subfigure}{0.12\textwidth}
			\includegraphics[width=\columnwidth]{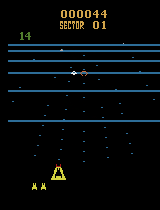}
		\end{subfigure}
		\begin{subfigure}{0.12\textwidth}
			\includegraphics[width=\columnwidth]{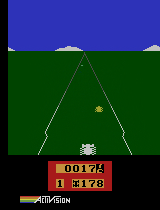}
		\end{subfigure}
		\begin{subfigure}{0.12\textwidth}
			\includegraphics[width=\columnwidth]{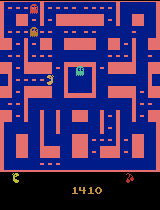}
		\end{subfigure}
		\begin{subfigure}{0.12\textwidth}
			\includegraphics[width=\columnwidth]{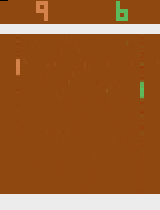}
		\end{subfigure}
		\begin{subfigure}{0.12\textwidth}
			\includegraphics[width=\columnwidth]{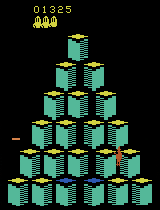}
		\end{subfigure}
		\begin{subfigure}{0.12\textwidth}
			\includegraphics[width=\columnwidth]{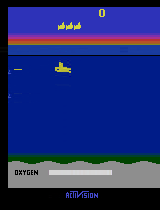}
		\end{subfigure}
		\begin{subfigure}{0.12\textwidth}
			\includegraphics[width=\columnwidth]{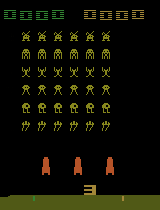}
		\end{subfigure}
		\caption{Original frames from the Atari dataset (top) along with their respective reconstructions (bottom) from the TT-cores compressed using TT-ICE with $\varepsilon_{des}=0.01$. There are no visible artifacts on the frames. ITTD5 reconstructions are not included here since ITTD5 fails at all games for $\varepsilon_{des}=0.01$. Left-to-right: Breakout, Beam Rider, Enduro, Ms.Pac-Man, Pong, Q*bert, Seaquest, Space Invaders 
		}
		\label{fig:eps001reconstruction}
	\end{figure}

\begin{table}[h!]
	\mcaption{Summary statistics for the other ATARI games. Game: name of the game, Machine: codename of the computer that conducted the experiments, Method: decomposition algorithm used (TT-ICE$^*$: Algorithm~\ref{alg:coreupdateheuristic}, TT-ICE: Algorithm~\ref{alg:TT-ICE}, ITTD: \cite[Alg 3]{Liu2018}), $\varepsilon$: the desired relative error bound for the given decomposition method, Frames: total number of the frames compressed in the TT-cores, Used: number of frames used to update the TT-cores, Time: total execution time of the algorithm, mean RRE: final mean RRE at the end of the experiment, CR: compression ratio. For each experiment, the best performance is shown in bold for Time, mean RRE, and CR. All failures are due to insufficient memory.} 
	\centering
	\begin{tabular}{l|l|c|c|c|c|c|c|c}
		Game & Machine & Method & $\varepsilon$ & Frames & Used & Time & mean RRE & CR\\
		\hline
		\multirow{3}{*}{Breakout}
		 	&\multirow{3}{*}{M2}&TT-ICE& 0.1 &9295&9295&212.2&0.078&2748.1\\
		  	&&TT-ICE$^*$& 0.1 &9295&2376&\textbf{58.9}&\textbf{0.084}&\textbf{3525.9}\\
		   	&&ITTD 5& 0.1 &9295&9295&801.2&0.083&66.55\\
		   	\hline
	   	\multirow{3}{*}{Breakout}
		   	&\multirow{3}{*}{M2}&TT-ICE& 0.01 &9295&9295&572.7&0.0075&9.7\\
		   	&&TT-ICE$^*$& 0.01 &9295&4997&\textbf{198.5}&\textbf{0.0078}&\textbf{10.1}\\
		   	&&ITTD 5& 0.01 &436&436&Fails&Fails&Fails\\
		\hline
		\hline
		\multirow{3}{*}{Beamrider}
			&\multirow{3}{*}{M2}&TT-ICE& 0.1 &42473&42473&2689.1&\textbf{0.089}&\textbf{130.7}\\
			&&TT-ICE$^*$& 0.1 &42473&4017&\textbf{193.0}&0.087&122.3\\
			&&ITTD 5& 0.1 &22606&22606&Fails&Fails&Fails\\
			\hline
		\multirow{3}{*}{Beamrider}
			&\multirow{3}{*}{M2}&TT-ICE& 0.01 &26257&26257&Fails&Fails&Fails\\
			&&TT-ICE$^*$& 0.01 &42473&22693&\textbf{2334.3}&\textbf{0.008}&\textbf{4.3}\\
			&&ITTD 5& 0.01 &4937&4937&Fails&Fails&Fails\\
		\hline
		\hline
		\multirow{3}{*}{Enduro}
			&\multirow{3}{*}{M1}&TT-ICE& 0.1 &132880&132880&11158.8&\textbf{0.096}&\textbf{246.6}\\
			&&TT-ICE$^*$& 0.1 &132880&10043&\textbf{940.1}&0.095&225.6\\
			&&ITTD 5& 0.1 &96329&96329&Fails&Fails&Fails\\
			\hline
		\multirow{3}{*}{Enduro}
			&\multirow{3}{*}{M1}&TT-ICE& 0.01 &49819&49819&Fails&Fails&Fails\\
			&&TT-ICE$^*$& 0.01 &76396&40482&Fails&Fails&Fails\\
			&&ITTD 5& 0.01 &13289&13289&Fails&Fails&Fails\\
		\hline
		\hline
		\multirow{3}{*}{Q*bert}
			&\multirow{3}{*}{M3}&TT-ICE& 0.1 &18742&18742&1063.6&0.089&817.7\\
			&&TT-ICE$^*$& 0.1 &18742&3993&\textbf{180.7}&\textbf{0.094}&\textbf{968.4}\\
			&&ITTD 5& 0.1 &18742&18742&1477.5&\textbf{0.094}&18.0\\
		\hline
		\multirow{3}{*}{Q*bert}
			&\multirow{3}{*}{M3}&TT-ICE& 0.01 &18742&18742&1740.4&0.006&13.6\\
			&&TT-ICE$^*$& 0.01 &18742&5272&\textbf{470.0}&\textbf{0.008}&\textbf{15.8}\\
			&&ITTD 5& 0.01 &2286&2286&Fails&Fails&Fails\\
		\hline
		\hline
		
		\multirow{3}{*}{Seaquest}
			&\multirow{3}{*}{M1}&TT-ICE& 0.1 &36805&36805&1793.8&0.091&1176.8\\
			&&TT-ICE$^*$& 0.1 &36805&3183&\textbf{173.1}&\textbf{0.096}&\textbf{1703.7}\\
			&&ITTD 5& 0.1 &36805&36805&2435.0&0.096&30.4\\
		\hline
		\multirow{3}{*}{Seaquest}
			&\multirow{3}{*}{M3}&TT-ICE& 0.01 &36805&36805&3907.0&0.007&3.1\\
			&&TT-ICE$^*$& 0.01 &36805&21194&\textbf{2024.6}&\textbf{0.008}&\textbf{3.4}\\
			&&ITTD 5& 0.01 &6667&6667&Fails&Fails&Fails\\
		\hline
		\hline
		
		\multirow{3}{*}{Pong}
			&\multirow{3}{*}{M1}&TT-ICE& 0.1 &96302&96302&1566.0&\textbf{0.066}&\textbf{99835.9}\\
            &                    &TT-ICE$^*$& 0.1 &96302& 2738&\textbf{70.0}&\textbf{0.066}&\textbf{99835.9}\\
            &                 &ITTD 5& 0.1 &96302&96302&1644.6&0.065&2234.6\\
		\hline
		\multirow{3}{*}{Pong}
			&\multirow{3}{*}{M1}&TT-ICE& 0.01 &96302&96302&12895.3&\textbf{0.009}&\textbf{48.3}\\
		    &                    &TT-ICE$^*$& 0.01 &96302&17219&\textbf{1283.2}&\textbf{0.009}&47.2\\
		    &                 &ITTD 5& 0.01 &46015&46015&Fails&Fails&Fails\\
		\hline
		\hline
	\end{tabular}
\label{tab:atariresults}
\end{table}
\section{Comparison with TT-FOA}
\label{app:ttfoacomparison}
This section includes the comparison study between TT-ICE${}^{*}$ and TT-FOA along with its discussion. In regimes where we found TT-FOA effective to run.
We were not able to deploy TT-FOA on our problems of interest because it required significantly greater storage and computational requirements than TT-FOA or ITTD due to large-scale auxiliary matrices. 
For example,  the TT-ranks used to approximate the tensors in \cite{Thanh2021} are very low in comparison to the final TT-ranks that we obtain in our numerical experiments. As a result, we performed a comparison using a similar testbed to that in the TT-FOA paper. To be able to provide a fair comparison between TT-ICE${}^{*}$ and TT-FOA, we implemented the Matlab script provided by authors of \cite{Thanh2021} in python\footnote{The accuracy of the python implementation was verified by comparing the results of our python implementation and the published Matlab script using the same 4-dimensional tensor streams. To time TT-FOA fairly in the following experiments, we removed the relative error computation that was included in the Matlab implementation of the algorithm and computed the relative error separately.} and benchmarked against TT-ICE${}^{*}$ using a synthetic 4-dimensional tensor with size $10\times15\times20\times500$ with TT-ranks $[1,2,3,5,1]$, which is same as the study conducted in \cite{Thanh2021}.

We considered 3 scenarios for TT-FOA: 1) TT-FOA initialized with the true TT-ranks, 2) TT-FOA initialized with underestimated TT-ranks, and 3) TT-FOA initialized with overestimated TT-ranks. On the other hand, for TT-ICE${}^{*}$ we set $\varepsilon_{des}=10^{-8}$ since the synthetic tensor is exactly low rank. We then constructed 20 random 4-dimensional tensor streams having increments of size $10\times15\times20$. 

\Cref{fig:ttfoacomparison} shows the relative error of the approximation at each increment step averaged over 20 repetitions of the experiment. A comparison of the compression ratio of methods is not provided here since the TT-ranks are assumed known from the beginning for TT-FOA.

\begin{figure}[!h]
	\begin{subfigure}{\textwidth}
		\centering
		\includegraphics[]{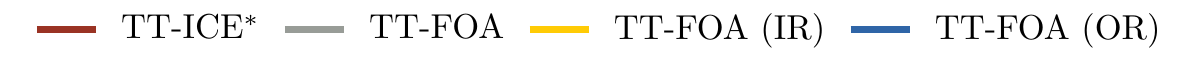}
	\end{subfigure}
	\begin{subfigure}{\textwidth}
		\centering
		\includegraphics[width=0.5\textwidth]{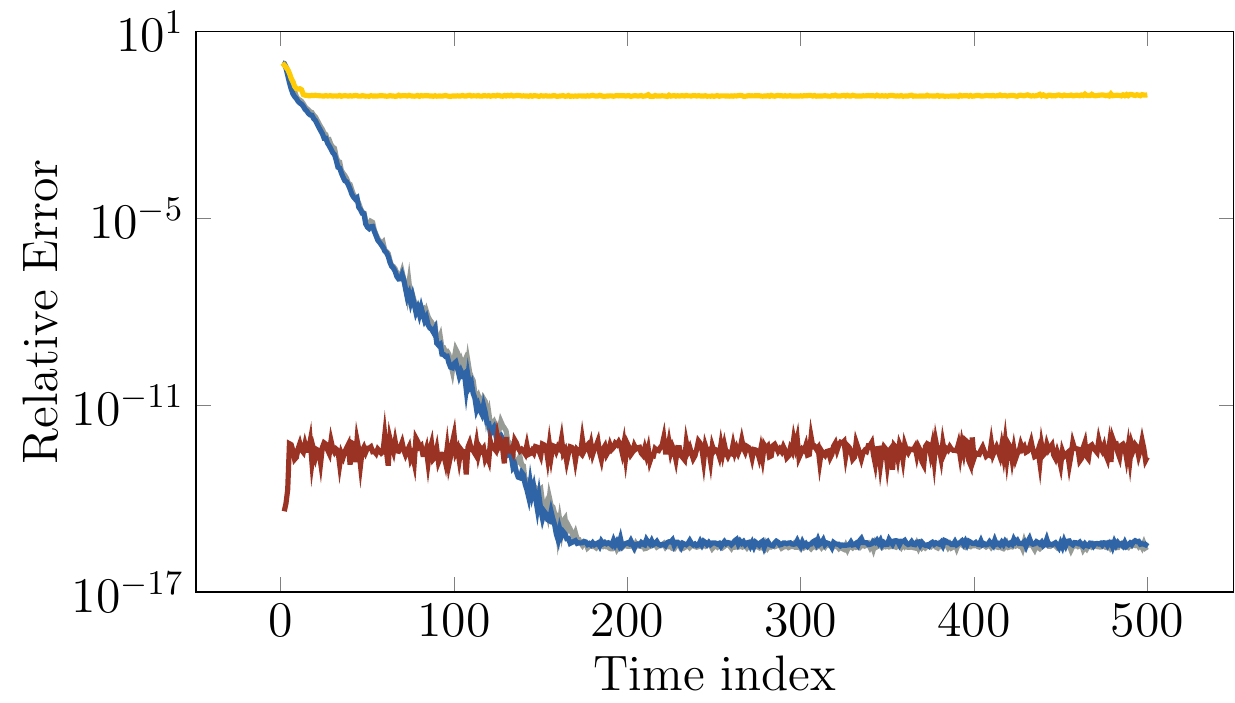}
	\end{subfigure}
	\caption{Comparison of relative errors of TT-ICE${}^{*}$ and TT-FOA\cite{Thanh2021} averaged over 20 repetitions. TT-ICE${}^{*}$: TT-ICE${}^{*}$ with $\varepsilon_{des}=10^{-8}$, TT-FOA: TT-FOA initialized with correct TT-ranks, TT-FOA (IR): TT-FOA initialized with insufficient TT-ranks, TT-FOA (OR): TT-FOA initialized with overestimated TT-ranks. Unlike TT-FOA, TT-ICE${}^{*}$ achieves a low relative error without the need for a convergence period and discovers the correct underlying TT-ranks. In case of underestimated TT-ranks, TT-FOA fails to achieve low relative error.}
	\label{fig:ttfoacomparison}
\end{figure}

\Cref{fig:ttfoacomparison} shows that TT-FOA requires a convergence period before reaching a stable relative error even when it is initialized with correct TT-ranks. In the case of underestimated TT-ranks, TT-FOA converges to a higher relative error depending on the difference in the estimated and actual TT-ranks. TT-FOA converges to the same level of relative error as the correctly estimated case when the TT-ranks are overestimated. This behavior is expected for this set of experiments since we have an exactly low-rank tensor stream. On the other hand, TT-ICE${}^{*}$ neither requires estimation of the TT-ranks beforehand, nor needs a convergence period to achieve the desired level of relative error.

The difference between both methods becomes evident when the time to compress the streams is considered. TT-FOA performs Kronecker products and matrix inversions, and the size of the matrices involved in those operations plays a decisive role in the speed of the algorithm. Therefore, there is a notable difference in computation time between different TT-rank estimations. When the correct TT-ranks are estimated ([1,2,3,5,1]), compression of the streams takes on average $1.8070s$. This time becomes $1.5543s$ when TT-ranks are underestimated as [1,1,2,3,1] and increases to $4.1290s$ when TT-ranks are overestimated as [1,3,5,10,1]. TT-ICE${}^{*}$ compresses the entire stream in $0.5806s$ on average.

When the same experiment is repeated with a tensor of the same size but with TT-ranks [1,4,10,30,1], TT-FOA requires $314.53s$ on average with correct TT-ranks. This time can get as low as $14.965s$ when TT-FOA is initialized with underestimated TT-ranks [1,3,5,20,1] and can get as high as $628.70s$ with overestimated TT-ranks [1,6,15,40,1]. On the other hand, TT-ICE${}^{*}$ scales very well in this scenario and compresses the entire stream in $0.6656s$ on average.

\end{appendices}

\end{document}